\documentclass[final,leqno,onefignum,onetabnum]{siamltex}

\usepackage{enumerate}
\usepackage[shortlabels]{enumitem}
\usepackage{amsmath}
\usepackage{amssymb}
\usepackage{multirow}
\usepackage[firstinits=true]{biblatex}

\usepackage{hyperref}

\usepackage{algorithm,algorithmic}

\usepackage{psfrag}
\usepackage{fullpage}
\usepackage[dvipsnames]{xcolor}
\usepackage{float}

\usepackage[normalem]{ulem}

\usepackage{multicol}
\usepackage{epstopdf}
\usepackage{mathtools}

\usepackage{dsfont}

\usepackage{subcaption}

\usepackage[process=auto, crop=pdfcrop]{pstool}

\usepackage{pgfplots}
\pgfplotsset{compat=newest}
\usetikzlibrary{plotmarks}
\usetikzlibrary{arrows.meta}
\usepgfplotslibrary{patchplots}
\usepackage{grffile}
\usepackage{amsmath}

\definecolor{bluefromheader}{cmyk}{.83,.44,0,.33}

\makeatletter
\newcommand{\mylabel}[2]{#2\def\@currentlabel{#2}\label{#1}}
\makeatother

\newtheorem{assumption}[theorem]{Assumption}
\newtheorem{remark}[theorem]{Remark}
\newtheorem{example}[theorem]{Example}

\newcommand{\stochDomain}{\Omega}
\newcommand{\sigmaAlgebra}{\cA}
\newcommand{\probSpace}{(\stochDomain, \sigmaAlgebra, \bbP)}
\newcommand{\timeInterval}{\bbT}
\newcommand{\spaceDomain}{\bbX}
\newcommand{\stopTime}{T}
\newcommand{\solFlux}{f}
\newcommand{\fluxFunction}{\mathfrak{\solFlux}}
\newcommand{\invFlux}{\fluxFunction^{-1}_{\pm}}
\newcommand{\spaceVar}{x}
\newcommand{\timeVar}{t}
\newcommand{\stochVar}{\omega}
\newcommand{\solution}{u}
\newcommand{\initialCond}{\solution_{0}}
\newcommand{\steadyStatePar}{\alpha}
\newcommand{\entropyConstant}{m}
\newcommand{\steadyStateSolution}[1]{\entropyConstant^{#1}_{\steadyStatePar}(\spaceVar)}
\newcommand{\steadyStateSolutions}{\entropyConstant^{\pm}_{\steadyStatePar}(\spaceVar)}
\newcommand{\stochSteadyStateSolutions}{\entropyConstant^{\pm}_{\steadyStatePar}(\stochVar, \spaceVar)}
\newcommand{\stochSteadyStateSolution}[1]{\entropyConstant^{#1}_{\steadyStatePar}}
\newcommand{\stochSteadyStateSolutionArg}[1]{\entropyConstant^{#1}_{\steadyStatePar}(\stochVar, \spaceVar)}
\newcommand{\partAdaptedEntropies}[1]{E_{#1}}
\newcommand{\nullSet}{\cN}
\newcommand{\subsolution}{g_{-}}
\newcommand{\supersolution}{g_{+}}
\newcommand{\modalityPointFct}{\solution_M}
\newcommand{\imageBound}{M_0}

\newcommand{\lipschitzInterval}{I}
\newcommand{\fluxLipschitzConstant}{L_{\lipschitzInterval}}

\newcommand{\solutionSpace}{L^{\infty}(\bbR \times \timeInterval) \cap C^{0}(\timeInterval; L^{1}_{\text{loc}}(\bbR))}

\newcommand{\jumpCoeff}{a}
\newcommand{\meanField}{\bar{\jumpCoeff}}
\newcommand{\gaussFunctional}{\phi}
\newcommand{\gaussField}{W}
\newcommand{\jumpField}{P}
\newcommand{\covarianceOperator}{Q}
\newcommand{\borelSigmaAlgebra}{\cB}
\newcommand{\partition}{\cT}
\newcommand{\jumpMeasure}{\lambda}
\newcommand{\domainOfInterest}{\spaceDomain}
\newcommand{\leftDomainBound}{\domainOfInterest_{l}}
\newcommand{\rightDomainBound}{\domainOfInterest_{r}}
\newcommand{\numberOfJumps}{\tau}
\newcommand{\jumpHeights}[1]{\jumpField_{#1}}
\newcommand{\coeffLowerBound}{\jumpCoeff_{-}}

\newcommand{\discontinuitySet}{\mathfrak{D}}
\newcommand{\jumpPosition}[1]{\chi_{#1}}

\newcommand{\f}{f}

\newcommand{\eigenValue}{\eta}
\newcommand{\eigenFunction}{e}
\newcommand{\KLRV}{Z}
\newcommand{\KLindex}{N}
\newcommand{\approxGaussField}{\gaussField^{\KLindex}_{\bbR}}
\newcommand{\maternAbbrev}{M}

\newcommand{\testFunction}{\varphi}
\newcommand{\variance}{\sigma^2}
\newcommand{\smoothness}{\nu}
\newcommand{\correlationLength}{\rho}
\newcommand{\besselFunction}{K_{\smoothness}}

\newcommand{\discretization}{\Delta}
\newcommand{\spaceGrid}{\spaceDomain_{\discretization}}
\newcommand{\gridPoint}[1]{\spaceVar_{#1}}
\newcommand{\numGridPoints}{N_{\spaceVar}}
\newcommand{\gridCell}[1]{\cX_{#1}}
\newcommand{\spaceIdx}{i}
\newcommand{\spaceStepSize}[1]{\discretization_{\spaceVar}^{#1}}
\newcommand{\timeGrid}{\timeInterval_{\discretization}}
\newcommand{\timePoint}[1]{\timeVar^{#1}}
\newcommand{\numTimePoints}{N_{\timeVar}}
\newcommand{\timeStepSize}[1]{\discretization_{\timeVar}^{#1}}

\newcommand{\timeIdx}{n}
\newcommand{\discontinuityInterface}[1]{d_{#1}}
\newcommand{\waveCell}[1]{\mathfrak{w}_{#1}}

\newcommand{\CFLconst}{L_u^n}

\newcommand{\approxSolution}[2]{\tilde{\solution}^{#1}_{#2}}
\newcommand{\approxFlux}{F}
\newcommand{\inFlux}[2]{\approxFlux^{#1}_{#2-\frac{1}{2}}}
\newcommand{\outFlux}[2]{\approxFlux^{#1}_{#2+\frac{1}{2}}}
\newcommand{\optVar}{\theta}

\newcommand{\DeclareAutoPairedDelimiter}[3]{%
	\expandafter\DeclarePairedDelimiter\csname Auto\string#1\endcsname{#2}{#3}%
	\begingroup\edef\x{\endgroup
		\noexpand\DeclareRobustCommand{\noexpand#1}{%
			\expandafter\noexpand\csname Auto\string#1\endcsname*}}%
	\x}

\DeclareAutoPairedDelimiter{\abs}{\lvert}{\rvert}
\DeclareAutoPairedDelimiter{\norm}{\lVert}{\rVert}
\DeclareAutoPairedDelimiter{\set}{\{}{\}}
\DeclareMathOperator{\sign}{sign}

\newcommand{\timeDeriv}{\frac{\partial}{\partial \timeVar}}
\newcommand{\spaceDeriv}{\frac{\partial}{\partial \spaceVar}}

\renewcommand{\d}[1]{\ensuremath{\operatorname{d}\!{#1}}}

\newcommand{\cA}{\mathcal{A}}
\newcommand{\cB}{\mathcal{B}}

\newcommand{\cD}{\mathcal{D}}
\newcommand{\cE}{\mathcal{E}}

\newcommand{\cN}{\mathcal{N}}

\newcommand{\cT}{\mathcal{T}}
\newcommand{\cU}{\mathcal{U}}

\newcommand{\cX}{\mathcal{X}}

\newcommand{\bbE}{\mathbb{E}}

\newcommand{\bbN}{\mathbb{N}}

\newcommand{\bbP}{\mathbb{P}}
\newcommand{\bbQ}{\mathbb{Q}}
\newcommand{\bbR}{\mathbb{R}}

\newcommand{\bbT}{\mathbb{T}}

\newcommand{\bbX}{\mathbb{X}}

\bibliography{./references.bib}

\mathtoolsset{showonlyrefs=true}

\title{Scalar conservation laws with stochastic discontinuous flux function}

\makeatletter
\newcommand{\specificthanks}[1]{\@fnsymbol{#1}}
\makeatother

\author{
	Lukas Brencher\textsuperscript{\specificthanks{2},}\thanks{Corresponding author. 
		(\href{mailto:lukas.brencher@mathematik.uni-stuttgart.de}{lukas.brencher@mathematik.uni-stuttgart.de})}
	\and
	Andrea Barth\thanks{University of Stuttgart, Allmandring 5b, 70569 Stuttgart, Germany}
}

\begin{document}
\maketitle
 
\begin{abstract}
	A variety of real-world applications are modeled via hyperbolic conservation laws.
	To account for uncertainties or insufficient measurements, random coefficients may be incorporated.
	These random fields may depend discontinuously on the state space, e.g., to represent permeability in a heterogeneous or fractured medium.
	We introduce a suitable admissibility criterion for the resulting stochastic discontinuous-flux conservation law and prove its well-posedness. 
	Therefore, we ensure the pathwise existence and uniqueness of the corresponding deterministic setting and present a novel proof for the measurability of the solution, since classical approaches fail in the discontinuous-flux case.
	As an example of the developed theory, we present a specific advection coefficient, which is modeled as a sum of a continuous random field and a pure jump field.
	This random field is employed in the stochastic conservation law, in particular a stochastic Burgers' equation, for numerical experiments.
	We approximate the solution to this problem via the Finite Volume method and introduce a new meshing strategy that accounts for the resulting standing wave profiles caused by the flux-discontinuities.
	The ability of this new meshing method to reduce the sample-wise variance is demonstrated in numerous numerical investigations.
\end{abstract}

\section{Introduction}

Conservation laws are widely used to model problems in various applications, such as fluid and gas dynamics \cite{Andreianov2014Approximating, Jaffre1996Numerical}, hydrodynamics \cite{Chen2008Hyperbolic}, nonlinear elasticity, semiconductor simulation and many more. We refer to the books of Dafermos \cite{Dafermos2000Hyperbolic} and LeVeque \cite{LeVeque1992Numerical} and the references therein for more applications. 
There are two natural extensions of models based on conservation laws: 
\begin{enumerate}[(i)]
	\item One may want to allow the flux function to depend (possibly discontinuously) on the state space, e.g., to model heterogeneities or fractures in a porous medium for subsurface or two-phase flow simulations \cite{Andreianov2014Approximating, Jaffre1996Numerical, Gimse1992Solution}.
	This approach is also used for many other models in a variety of fields, e.g.,  vehicular traffic flows \cite{Andreianov2010Finite, Coclite2005Traffic}, statistical mechanics \cite{Chen2008Hyperbolic}, sedimentation \cite{Buerger2004Well} and kidney physiology \cite{Perthame2014simple}.
	\item One might be interested in including stochastic terms, e.g., a stochastic flux function \cite{Holden1997Conservation, Lions2013Scalar, Lions2014Scalar} or a stochastic source \cite{Karlsen2017stochastic, Kim2003stochastic}, to account for uncertainties or insufficient data in the underlying model.
\end{enumerate}
To combine both extensions, we consider a stochastic conservation law, where the flux function depends discontinuously on the state space. 
In this paper, we are concerned with well-posedness of the solution to such a stochastic conservation law. 
A special case of the considered stochastic conservation law are flux functions, in which a discontinuous random coefficient is incorporated. 
This random coefficient is constructed via a continuous part and a discontinuous part, which is inspired by the unique characterization of Lévy processes via the Lévy-Khintchine formula \cite{Applebaum2009Levy, Sato1999Levy}. By this formula, every Lévy process can be seen as the composition of three independent components, namely, a drift term, a Brownian motion and a pure jump process.

The paper is structured as follows: 
In the remainder of this section, we discuss existing results on conservation laws with discontinuous flux function in Section \ref{ssec:intro_discontinuous_CLs}, before stating existing results on stochastic conservation laws in Section \ref{ssec:intro_SCLs}.
In the following section, we state the considered problem and show our main result: the well-posedness of pathwise solutions.
In Section \ref{sec:stochJumpCoeff}, we define a random coefficient and its numerical approximation.
Afterwards, we introduce a numerical discretization of the stochastic conservation law with discontinuous random coefficient and investigate the convergence of the approximations in numerical experiments.

\subsection{Conservation laws with spatial discontinuities}
\label{ssec:intro_discontinuous_CLs}

We are interested in scalar conservation laws with a flux that depends discontinuously on the spatial variable.
For $t \in \bbT$, where $\bbT \subset \bbR$ denotes the time interval, the conservation law is given by 
\begin{equation}
	\label{eq:deterministicConservationLaw}
	\begin{aligned}
		\solution_{\timeVar} + \fluxFunction(\spaceVar, \solution)_{\spaceVar} &= 0 &&\text{in } \bbR \times \timeInterval \ , \\
		\solution(\spaceVar, 0) &= \initialCond (\spaceVar) &&\text{in } \bbR \ ,
	\end{aligned}
\end{equation}
where $\initialCond (\spaceVar) \in L^{\infty}(\bbR)$ denotes the initial condition of the partial differential equation.
Furthermore, the flux $\fluxFunction$ is assumed to depend discontinuously on space. 
In case the flux function $\fluxFunction$ is sufficiently smooth, well-posedness of \eqref{eq:deterministicConservationLaw} is well established \cite{Bouchut1998Krukovs, Kruzkov1970First}.
However, even with smooth initial data and a sufficiently smooth flux function, the solution might develop discontinuities.
Therefore, weak solutions are sought, but they are not unique without an additional admissibility criterion selecting the physical solution.
In case of a smooth space dependency of the flux $\fluxFunction$ the entropy condition was described by Kru\v{z}kov \cite{Kruzkov1970First}.
A function $\solution \in L^{\infty}(\bbR \times \timeInterval; \bbR)$ is an entropy solution of \eqref{eq:deterministicConservationLaw} if for all $\entropyConstant \in \bbR$ it satisfies the following inequality in the distributional sense:
\begin{equation}
	\label{eq:KruzkovEntropyCondition}
	\timeDeriv \abs{\solution - \entropyConstant} + \spaceDeriv \left( \sign(\solution - \entropyConstant) (\fluxFunction(\spaceVar, \solution) - \fluxFunction(\spaceVar, \entropyConstant)) \right) + \sign(\solution - \entropyConstant) \spaceDeriv \fluxFunction(\spaceVar, \entropyConstant) \leq 0 \ .
\end{equation}
For discontinuous flux functions, the Kru\v{z}kov entropy condition \eqref{eq:KruzkovEntropyCondition} does not make sense due to the term $\sign(\solution - \entropyConstant) \spaceDeriv \fluxFunction(\spaceVar, \entropyConstant)$.
To overcome this difficulty, a variety of criteria were introduced to select the appropriate weak solution. 
However, for some problems, different selection criteria might lead to different solutions.
Therefore, the solution we are seeking might depend on the context.

A common approach to define a selection criterion is to require the usual entropy condition, when the flux is continuous, and impose additional constraints at the discontinuities (see, e.g., \cite{Adimurthi2005Godunov, Bulicek2011scalar, Ghoshal2011Existence}).
A natural choice of such a condition is the continuity of the flux at the heterogeneities,  given by the Rankine-Hugoniot condition
\begin{equation}
	\label{eq:Rankine-Hugoniot}
	\fluxFunction(x^-, u(x^-, t)) = \fluxFunction(x^+, u(x^+, t)) \quad t \in \bbT .
\end{equation}
However, this condition is not sufficient to select an unique solution.
Note that for discontinuous flux functions this procedure requires the existence of traces along the discontinuities, which implies the spatial points to occur at isolated points.
For an exhaustive review of different selection criteria, the reader is referred to \cite{Andreianov2011theory}, where each existing admissibility criterion is associated to a so-called \textit{Germ}.

%
%

\subsubsection{Audusse-Perthame adapted entropy formulation}

The theory of adapted-entropy solutions for conservation laws was introduced by Audusse and Perthame \cite{Audusse2005Uniqueness} in 2005, based on the ideas of Baiti and Jenssen \cite{Baiti1997Well} from 1997.
The provided framework is formulated, such that no additional interface conditions are imposed. 
As a result, the existence of traces is not required, since the discontinuity points are not distinguished separately. This makes it possible to include an infinite number of discontinuity points in the flux function.

The main idea of the framework is to adapt the classical Kru\v{z}khov entropy condition \eqref{eq:KruzkovEntropyCondition} in such a way that the troublesome term vanishes.
Therefore, instead of the classical Kru\v{z}khov \cite{Kruzkov1970First} entropies $(\abs{\solution - \entropyConstant})_{\entropyConstant \in \bbR}$, adapted entropies $(\abs{\solution - \steadyStateSolution{}})_{\steadyStatePar \in \bbR}$ are considered.
Here, $\steadyStateSolution{}$ denotes the unique solution of the problem
\begin{equation}
	\label{eq:steadyStateFlux}
	\fluxFunction(\spaceVar, \steadyStateSolution{}) = \steadyStatePar \quad \text{for a.e. }\spaceVar \in \bbR \ .
\end{equation}
The existence and uniqueness of $\steadyStateSolution{}$ in \eqref{eq:steadyStateFlux} is guaranteed by the following assumptions:
\begin{assumption}[Assumption on the flux function: adapted entropy formulation]
	\label{as:deterministicFlux}
	\begin{enumerate}[label=(A-\arabic*)]
		\item \label{as:A-1} The flux function $\fluxFunction(\spaceVar, \rho)$ is continuous at all points of $(\bbR \setminus \nullSet) \times \bbR$, where $\nullSet \subset \bbR$ is a closed set of measure zero that contains the discontinuity points of $\fluxFunction(\cdot, \rho)$.
		\item \label{as:A-2} There exist two functions $\subsolution, \supersolution \in C^{0}(\bbR)$ such that for all $\spaceVar \in \bbR$ it holds that $\subsolution (\rho) \leq |\fluxFunction(\spaceVar, \rho)| \leq \supersolution (\rho)$, where $\subsolution$ is a non-negative (non-strictly) decreasing then increasing function which satisfies $|\subsolution(\pm \infty)| = +\infty$.
		\item \label{as:A-3} There exists a function $\modalityPointFct (x): \bbR \rightarrow \bbR$ and a constant $\imageBound$ such that for $x \in \bbR \setminus \nullSet$, $\fluxFunction(\spaceVar, \cdot)$ is a locally Lip\-schitz one-to-one function from $(-\infty, \modalityPointFct(x)]$ and $[\modalityPointFct(x), +\infty)$ to $[\imageBound, +\infty)$ (or $(-\infty, \imageBound]$) that satisfies $\fluxFunction( \spaceVar, \modalityPointFct(x)) = 0$ and with common Lipschitz constant $L_I$ for all $x \in \bbR \setminus \nullSet$ and all $\solution \in I$, where $I \subset \bbR$ is any bounded interval.
	\end{enumerate}
	Alternatively, instead of \ref{as:A-3} we may consider the assumption
	\begin{enumerate}
		\item[\mylabel{as:A-3alt}{(A-3')}] For $\spaceVar \in \bbR \setminus \nullSet$, the flux function $\fluxFunction (\spaceVar, \cdot)$ is a locally Lipschitz one-to-one function from $\bbR$ to $\bbR$ with common Lipschitz constant $L_I$ for all $x \in \bbR \setminus \nullSet$ and all $\solution \in I$, where $I \subset \bbR$ is any bounded interval.
	\end{enumerate}
\end{assumption}

In case the flux function satisfies Assumptions \ref{as:A-1}\--\ref{as:A-3}, for any constant $\steadyStatePar \in [\imageBound, \infty)$ (or $\alpha \in (-\infty, \imageBound]$), there exist two steady-state solutions $\steadyStateSolution{+}$ from $\bbR$ to $[\modalityPointFct (x), \infty)$ and $\steadyStateSolution{-}$ from $\bbR$ to $(-\infty, \modalityPointFct (x)]$ of \eqref{eq:deterministicConservationLaw} satisfying \eqref{eq:steadyStateFlux}.
In case of Assumptions \ref{as:A-1}\--\ref{as:A-3alt}, the two steady state solutions coincide, which is even simpler.

Now, the adapted entropies $(\abs{\solution - \steadyStateSolution{}})_{\steadyStatePar \in \bbR}$ allow us to define adapted entropy solutions.
\begin{definition}[adapted entropy solution]
	\label{def:adaptedEntropySolution}
	A function $\solution \in L^{\infty}(\bbR \times \timeInterval) \cap C^{0}(\timeInterval, L^{1}_{\text{loc}}(\bbR))$ is an adapted entropy solution of \eqref{eq:deterministicConservationLaw} on $\bbR \times \timeInterval$, provided that for each $\steadyStatePar \in [\imageBound, \infty)$ (or $\alpha \in (-\infty, \imageBound]$) and the corresponding two steady-state solutions $\steadyStateSolutions$ of \eqref{eq:deterministicConservationLaw}, the following inequality holds in the sense of distributions:
	\begin{equation}
		\label{eq:adaptedEntropyCondition}
		\timeDeriv \abs{\solution(\spaceVar, \timeVar) - \steadyStateSolutions} + \spaceDeriv \left( \sign \left( \solution(\spaceVar, \timeVar) - \steadyStateSolutions \right) (\fluxFunction(\spaceVar, \solution(\spaceVar, \timeVar)) - \fluxFunction(\spaceVar, \steadyStateSolutions)) \right) \leq 0 \ .
	\end{equation}
\end{definition}

This formulation guarantees the uniqueness of solutions, since the \textit{doubling of variables} approach can be applied.
The existence was shown in \cite{Chen2008Hyperbolic} via the reduction of measure-valued solutions.
However, to the best of the author's knowledge, extending the adapted entropy formulation to multidimensional conservation laws is still an open problem.

\subsection{Stochastic conservation laws}
\label{ssec:intro_SCLs}

The field of stochastic conservation laws has attracted considerable attention in the past 20 years.
The concept of stochastic weak and entropy solutions was initially introduced by Holden and Risebro in \cite{Holden1997Conservation}, who considered nonlinear hyperbolic problems with time-dependent stochastic source terms.
This concept was further developed by Kim in \cite{Kim2003stochastic} and extended to conservation laws with multiplicative noise, see \cite{Feng2008Stochastic} for Gaussian and \cite{Biswas2015Conservation} for L\'{e}vy noise.
In \cite{Bauzet2012Cauchy}, Bauzet et al. used the concept of Kru\v{z}kov's entropy formulation and measure-valued solutions to establish existence and uniqueness of entropy solutions for multi-dimensional nonlinear conservation laws with multiplicative stochastic perturbation.
Recently, Karlsen and Storr{\o}sten generalized this approach by allowing the stochastic Kru\v{z}kov entropy condition to contain Malliavin differentiable random variables \cite{Karlsen2017stochastic}.
They also provided existence and uniqueness results for this extended framework.

In 2013 Lions \cite{Lions2013Scalar} considered scalar conservation laws with rough stochastic fluxes and extended this in \cite{Lions2014Scalar} to the case, where the rough stochastic flux function is space dependent.
The field of scalar conservation laws with rough fluxes combined with stochastic forcing was also considered by Hofmanov\'{a} in \cite{Hofmanova2016Scalar}, who proved well-posedness for the corresponding kinetic formulation and its solution.

Stochastic systems of conservation laws were considered by Mishra and Schwab in \cite{Mishra2012Sparse}, who introduced the uncertainty via random initial data.
In \cite{Poette2009Uncertainty} Po{\"e}tte et al. considered hyperbolic systems with a randomly perturbed flux function and showed the hyperbolicity of a stochastic Galerkin representation.
The stochastic Galerkin representation was later extended by \cite{Tryoen2012Adaptive} and \cite{Buerger2014hybrid}.
A comprehensive discussion on current numerical methods for hyperbolic systems of conservation laws can be found in \cite{Abgrall2017Uncertainty}.

In \cite{Mishra2016Numerical}, Mishra et al. considered the numerical approximation of scalar conservation laws in several space dimensions with random flux functions and also showed convergence analysis for the multilevel Monte Carlo finite volume method \cite{Mishra2012Sparse, Mishra2012Multi, Mishra2012Multilevel}.
Another numerical approximation was introduced by Risebro et. al. \cite{Risebro2016Multilevel}, who proposed a multilevel Monte Carlo front tracking algorithm for stochastic scalar conservation laws with bounded random flux function.
This approach was generalized in \cite{Koley2017multilevel} to the case of random degenerate convection-diffusion equations.

Recently, Barth and Stein \cite{Barth2019stochastic} considered semilinear hyperbolic stochastic partial differential equations, which are driven by L\'{e}vy noise and provided a fully discrete numerical approximation for this problem.

\section{Stochastic scalar conservation laws with discontinuous flux functions}

Let $\probSpace$ be a complete probability space and let $\timeInterval := [0, \stopTime] \subset \bbR_{\geq 0}$ be a time interval for some $\stopTime \in \bbR_{>0}$.
We consider a scalar, one-dimensional conservation law of the form
\begin{equation}
	\label{eq:stochasticConservationLaw}
	\begin{aligned}
		\solution_{\timeVar} + \fluxFunction(\stochVar, \spaceVar, \solution)_{\spaceVar} &= 0 &&\text{in } \stochDomain \times \bbR \times \timeInterval \ , \\
		\solution(\stochVar, \spaceVar, 0) &= \initialCond (\stochVar, \spaceVar) &&\text{in } \stochDomain \times \bbR \ ,
	\end{aligned}
\end{equation}
where $\initialCond \in L^p(\stochDomain, L^{\infty}(\bbR))$ denotes the stochastic initial condition of the partial differential equation.
We impose the following assumptions on the stochastic flux function:

\begin{assumption}[Pathwise assumption on the stochastic flux function]
	\label{as:stochFlux}
	We assume the stochastic discontinuous flux function to satisfy the following assumptions for every $\stochVar \in \stochDomain$:
	\begin{enumerate}[label=(B-\arabic*)]
		\item \label{as:B-1} The flux function $\fluxFunction(\stochVar, \cdot, \cdot)$ is continuous at all points of $(\bbR \setminus \nullSet(\stochVar)) \times \bbR$ with $\nullSet(\stochVar) \subset \bbR$ a closed set of measure zero that contains the discontinuity points of $\fluxFunction (\stochVar, \cdot, \rho)$.
		\item \label{as:B-2} There exist two functions $\subsolution (\stochVar, \cdot), \supersolution (\stochVar, \cdot) \in C^{0}(\bbR)$ such that for all $\spaceVar \in \bbR$ it holds that $\subsolution (\stochVar, \rho) \leq |\fluxFunction(\stochVar, \spaceVar, \rho)| \leq \supersolution (\stochVar, \rho)$, where $\subsolution (\stochVar, \cdot)$ is a decreasing then increasing function with $\subsolution(\stochVar, \rho) \geq 0$ and $|\subsolution(\stochVar, \pm \infty)| = +\infty$.
		\item \label{as:B-3} There exists a function $\modalityPointFct (\stochVar, \spaceVar): \stochDomain \times \bbR \rightarrow \bbR$ and a constant $\imageBound \in \bbR$ such that for  $x \in \bbR \setminus \nullSet(\stochVar)$, the flux function $\fluxFunction(\stochVar, \spaceVar, \cdot)$ is locally Lipschitz and one-to-one from $(-\infty, \modalityPointFct (\stochVar, \spaceVar)]$ and $[\modalityPointFct(\stochVar, \spaceVar), +\infty)$ to $[\imageBound, +\infty)$ (or $(-\infty, \imageBound]$). Further it satisfies $\fluxFunction(\stochVar, \spaceVar, \modalityPointFct(\stochVar, \spaceVar)) = \imageBound$ with common Lipschitz constant $\fluxLipschitzConstant$ for all $\spaceVar \in \bbR \setminus \nullSet(\stochVar)$ and all $\solution \in \lipschitzInterval$, where $\lipschitzInterval \subset \bbR$ is any bounded interval.
	\end{enumerate}
	
	Alternatively, instead of \ref{as:B-3} we may consider the assumption
	\begin{enumerate}
		\item[\mylabel{as:B-3alt}{(B-3')}] For $\spaceVar \in \bbR \setminus \nullSet (\stochVar)$, the flux function $\fluxFunction (\stochVar, \spaceVar, \cdot)$ is locally Lipschitz and one-to-one from $\bbR$ to $\bbR$ with common Lipschitz constant $\fluxLipschitzConstant$ for all $\spaceVar \in \bbR \setminus \nullSet(\stochVar)$ and all $\solution \in \lipschitzInterval$, where $\lipschitzInterval \subset \bbR$ is any bounded interval.
	\end{enumerate}
\end{assumption}

\begin{example}[Multiplicative flux]
	\label{ex:multiplicativeFlux}
	
	One example of stochastic flux functions satisfying Assumption \ref{as:stochFlux} is
	\begin{equation}
		\label{eq:multiplicativeFlux}
		\fluxFunction(\stochVar, \spaceVar, \rho) = \jumpCoeff (\stochVar, \spaceVar) \solFlux (\rho) .
	\end{equation}
	Here, $\jumpCoeff(\stochVar, \spaceVar)$ is continuous in $\spaceVar \in \bbR$ except on a closed set $\nullSet$ of measure zero that might depend on $\stochVar \in \stochDomain$ (assuring Assumption \ref{as:B-1}).
	Further, the coefficient $\jumpCoeff(\stochVar, \spaceVar)$ has positive spatially bounded paths, i.e., $0 < a_{-}(\stochVar) \leq \jumpCoeff(\stochVar, \spaceVar) \leq a_{+}(\stochVar) < \infty$ for some constants $a_{-}$ and $a_{+}$ that might depend on $\stochVar \in \stochDomain$ (assuring Assumption \ref{as:B-2}).
	The flux function $\solFlux (\rho)$ is locally Lipschitz continuous and is either monotone (Assumption \ref{as:B-3alt}) or convex (or concave) with $\solFlux(\modalityPointFct (\stochVar, \spaceVar)) = 0$ for some $\modalityPointFct$, in which case we have $\imageBound = 0$ (Assumption \ref{as:B-3}).
\end{example}

\begin{remark}[Existence of stochastic steady state solutions]
	\label{rem:stochFluxAssumptions}
	
	To define a meaningful entropy solution, we have to ensure the well-posedness of the stochastic steady state equation
	\begin{equation}
		\label{eq:stochSteadyState}
		\fluxFunction(\stochVar, \spaceVar, \stochSteadyStateSolutions) = \steadyStatePar \quad \text{for a.e. }\spaceVar \in \bbR \ .
	\end{equation}
	If assumptions \ref{as:B-1}\--\ref{as:B-3} are satisfied, then for any constant $\steadyStatePar \in [\imageBound, \infty)$ (or $\steadyStatePar \in (-\infty, \imageBound]$) and any $\stochVar \in \stochDomain$, there exist two steady state solutions $\stochSteadyStateSolution{+}: \stochDomain \times \bbR \rightarrow [\imageBound, \infty)$ and $\stochSteadyStateSolution{-}: \stochDomain \times \bbR \rightarrow (- \infty, \imageBound]$ of \eqref{eq:stochasticConservationLaw} satisfying \eqref{eq:stochSteadyState}.
	For the case \ref{as:B-1}\--\ref{as:B-3alt} the two steady state solutions coincide, i.e., $\stochSteadyStateSolutionArg{+} = \stochSteadyStateSolutionArg{-}$.
\end{remark}

The steady state solutions defined in Remark \ref{rem:stochFluxAssumptions} allow us to introduce the partially adapted Kru\v{z}khov entropies $\partAdaptedEntropies{\steadyStatePar}(\stochVar, \spaceVar, \solution) = |\solution - \stochSteadyStateSolutions|$, which lead to the notion of pathwise adapted entropy solutions.
\begin{definition}[Pathwise adapted entropy solution]
	Let $\stochVar \in \stochDomain$ be fixed.
	A function $\solution(\stochVar, \cdot, \cdot) \in L^{\infty}(\timeInterval \times \bbR) \cap C^{0}(\timeInterval, L^{1}_{\text{loc}}(\bbR))$ is an adapted entropy solution of \eqref{eq:stochasticConservationLaw} on $\bbR \times \timeInterval$, provided that for each $\steadyStatePar \in [\imageBound, \infty)$ (or $\steadyStatePar \in (-\infty, \imageBound]$) and the corresponding two steady-state solutions $\stochSteadyStateSolutions$ of \eqref{eq:stochSteadyState}, the inequality
	\begin{equation}
		\label{eq:adaptedEntropyCondition_integralForm}
		\begin{aligned}
			&\int \Big| \solution(\stochVar, \spaceVar, \timeVar) - \stochSteadyStateSolutions \Big| \timeDeriv \phi(x, t) \d t \d x \\
			&\hspace*{1.5cm} + \int \sign \Big( \solution(\stochVar, \spaceVar, \timeVar) - \stochSteadyStateSolutions \Big) \Big( \fluxFunction(\stochVar, \spaceVar, \solution(\stochVar, \spaceVar, \timeVar)) - \steadyStatePar \Big) \spaceDeriv \phi(x, t) \d t \d x \\
			&\hspace*{3cm} + \int \Big| \initialCond(\stochVar, \spaceVar) - \stochSteadyStateSolutions \Big| \phi (\spaceVar, 0) \d x \geq 0
		\end{aligned}
	\end{equation}
	holds for every test function $\phi \in C^{\infty}_{c}(\bbR \times \bbT; \bbR_{\geq 0})$.
\end{definition}

As a direct consequence of Assumption \ref{as:stochFlux} and the definition of pathwise adapted entropy solutions, we can conclude the following lemma.
\begin{lemma}
	\label{lem:fluxBound}
	Let Assumption \ref{as:B-2} be satisfied and fix an arbitrary $\stochVar \in \stochDomain$. Furthermore, suppose $\solution(\stochVar, \cdot, \cdot) \in L^{\infty}(\timeInterval \times \bbR) \cap C^{0}(\timeInterval, L^{1}_{\text{loc}}(\bbR))$ is a pathwise adapted entropy solution of \eqref{eq:stochasticConservationLaw}.
	Then the following estimate holds for a.e. $\spaceVar \in \bbR$ and a.e. $\timeVar \in \timeInterval$:
	\begin{equation}
		\label{eq:fluxBound}
		\abs{\fluxFunction (\stochVar, \spaceVar, \solution(\stochVar, \spaceVar, \timeVar))} \leq \max\limits_{-\norm{\solution(\stochVar, \cdot, \cdot)}_{L^{\infty}} \leq \sigma \leq \norm{\solution(\stochVar, \cdot, \cdot)}_{L^{\infty}}} \supersolution (\stochVar, \sigma) =: M(\stochVar) < \infty \ .
	\end{equation}
\end{lemma}

\begin{proof}
	By hypothesis, we have, for fixed $\stochVar \in \stochDomain$, that $\solution (\stochVar, \cdot, \cdot) \in L^{\infty}(\timeInterval \times \bbR)$. 
	Therefore, the maximization is over a compact interval.
	Further, by assumption \ref{as:B-2}, we have the existence of an upper bound $\supersolution$ of $\fluxFunction$, which is continuous in the second argument.
	The maximization is over a compact interval, hence we obtain the existence of the maximum.
	Also, with the same argument, we have $M (\stochVar) < \infty$, which concludes the proof.
\end{proof}

We are now able to state the following pathwise existence and uniqueness result:
\begin{theorem}
	\label{thm:stochExistenceUniqueness}
	Let the flux function $\fluxFunction$ satisfy Assumption \ref{as:stochFlux}.
	Then, for every $\stochVar \in \stochDomain$, there exists a unique pathwise adapted entropy solution $\solution (\stochVar, \cdot, \cdot) \in L^{\infty}(\timeInterval \times \bbR) \cap C^{0}(\timeInterval, L^{1}_{\text{loc}}(\bbR))$ of the initial value problem \eqref{eq:stochasticConservationLaw}, which satisfies for a.e. $t \in \bbT$:
	\begin{equation}
		\label{eq:stochUniqueness}
		\int_{a}^{b} \abs{\solution(\stochVar, \spaceVar, \timeVar)} \d{x} \leq \int_{a-M(\stochVar)\timeVar}^{b+M(\stochVar)\timeVar}  \abs{\initialCond(\stochVar, \spaceVar)} \d{x} \ .
	\end{equation}
	Here, $M (\stochVar)$ is defined as in \eqref{eq:fluxBound} and $a, b \in \bbR$ are arbitrary with $a < b$.
\end{theorem}

\begin{proof}
	Let $\stochVar \in \stochDomain$ be fixed. 
	The existence of a pathwise adapted entropy solution is proved similarly as for the deterministic case, see \cite{Chen2008Hyperbolic}.
	The proof is based on the reduction of measure-valued solutions to entropy solutions in $L^{\infty}$ and the consideration of a mollified version of the problem. \\
	For the uniqueness, which was shown in \cite{Audusse2005Uniqueness}, the use of adapted entropies allows the argumentation via the "doubling of variables" procedure of Kru\v{z}khov \cite{Kruzkov1970First}.
	Then, estimate \eqref{eq:stochUniqueness} is an immediate consequence of the uniqueness result.
\end{proof}

Note, the uniqueness result in Theorem \ref{thm:stochExistenceUniqueness} immediately implies that $\solution (\stochVar, \cdot, \timeVar) \in L^1(\bbR)$ for fixed $\timeVar \in \timeInterval$, if $\initialCond (\stochVar, \cdot) \in L^1(\bbR)$.
Furthermore, since we assumed $\initialCond \in L^{p}(\stochDomain; L^{\infty}(\bbR))$, we immediately obtain the existence of the $p$-th moment of the stochastic entropy solution via Inequality \eqref{eq:stochUniqueness}.

\subsection{Measurability of the solution}

In Theorem \ref{thm:stochExistenceUniqueness}, we have shown the existence and uniqueness of pathwise adapted entropy solutions. 
To complete the well-posedness investigation, it remains to show the measurability of the solution map $\solution: \stochDomain \rightarrow \solutionSpace$.
The following assumptions will allow us to prove the measurability of the solution map.
\begin{assumption}[Measurability of the stochastic flux function]
	\label{as:stochFlux_measurability}
	\begin{enumerate}[label=(B-\arabic*)]
		\item[\mylabel{as:B-4}{(B-4)}] For almost every $\spaceVar \in \bbR$, the mapping $\stochVar \mapsto \fluxFunction(\stochVar, \spaceVar, \rho)$ is measurable.
		\item[\mylabel{as:B-5}{(B-5)}] For almost every $\spaceVar \in \bbR$, the mapping $\stochVar \mapsto \invFlux(\stochVar, \spaceVar, \rho)$ is measurable, where $\invFlux(\stochVar, \spaceVar, \cdot)$ is the inverse function of $\fluxFunction(\stochVar, \spaceVar, \cdot)$.
		\item[\mylabel{as:B-6}{(B-6)}] The spatial paths of $\invFlux$ are bounded in the sense that, for fixed $\stochVar \in \stochDomain, \rho \in \bbR$, we have $\invFlux(\stochVar, \cdot, \rho) \in L^{\infty}(\bbR)$.
	\end{enumerate}
\end{assumption}
\begin{remark}
	\label{rem:measurabilityAssumption}
	
	The existence of the function $\invFlux$ is already guaranteed by Assumption \ref{as:stochFlux}, see also Remark \ref{rem:stochFluxAssumptions}.
\end{remark}
\begin{example}[Measurability of multiplicative flux]
	\label{ex:multiplicativeFlux_measurability}
	
	Considering again the flux defined in Example \ref{ex:multiplicativeFlux}, the Assumption \ref{as:B-4} reduces to the following: 
	For almost every $x \in \bbR$, the mapping $\stochVar \mapsto \jumpCoeff (\stochVar, \spaceVar)$ is measurable. \\
	For Assumption \ref{as:B-5}, we have
	\begin{equation}
		\invFlux(\stochVar, \spaceVar, \rho) = \frac{1}{\jumpCoeff(\stochVar, \spaceVar)} \solFlux^{-1}(\rho) .
	\end{equation}
	Since $\solFlux^{-1}(\rho)$ is continuous, it is especially measurable and thus, $\invFlux(\stochVar, \spaceVar, \rho)$ is measurable as the composition of two measurable functions.
	Consequently, for multiplicative flux functions satisfying Assumptions \ref{as:B-1}\--\ref{as:B-2} and \ref{as:B-3} or \ref{as:B-3alt}, together with Assumption \ref{as:B-4}, Assumption \ref{as:B-5} is automatically satisfied.
	Finally, since we have $0 < a_{-}(\stochVar) \leq \jumpCoeff(\stochVar, \spaceVar) \leq a_{+}(\stochVar) < \infty$, the spatial paths of $\invFlux$ are bounded.
\end{example}

To show the measurability of the solution map $\solution: \stochDomain \rightarrow \solutionSpace$, we will need the following \textit{adapted entropy functional}.
For the sake of notational simplicity, we define the function space $\cX \coloneqq \solutionSpace$.
Note that $\cX$ is separable, since $L^{1}_{\text{loc}}(\bbR)$ is separable (due to $L^{1}(\bbR)$ being separable and a dense subspace of $L^{1}_{\text{loc}}(\bbR)$) and $\timeInterval$ is compact.
\begin{definition}[Adapted entropy functional]
	\label{def:entropyFunctional}
	
	Let Assumption \ref{as:stochFlux} be satisfied. 
	Further, let $\alpha \in [\imageBound, \infty)$ (or $(-\infty, \imageBound]$) be fixed with corresponding steady state solutions $\stochSteadyStateSolution{\pm}$ defined by Equation \eqref{eq:stochSteadyState}.
	Then, the adapted entropy functional $J^{\alpha}: \Omega \times \cX \rightarrow \bbR$ is defined as
	\begin{subequations}
		\label{eq:entropyFunctional}
		\begin{align}
			(\stochVar, \nu) &\mapsto \int \abs{\nu(\spaceVar, \timeVar) - \stochSteadyStateSolutions} \timeDeriv \phi(x, t) \d t \d x \label{eq:entropyIntegral-1}\\
			&\qquad+ \int \sign (\nu(\spaceVar, \timeVar) - \stochSteadyStateSolutions) (\fluxFunction(\stochVar, \spaceVar, \nu(\spaceVar, \timeVar)) - \steadyStatePar) \spaceDeriv \phi(x, t) \d t \d x \label{eq:entropyIntegral-2} \\
			&\qquad\qquad + \int \abs{\initialCond(\stochVar, \spaceVar) - \stochSteadyStateSolutions} \phi (\spaceVar, 0) \d x \label{eq:entropyIntegral-3} ,
		\end{align}
	\end{subequations}
	for some test function $\phi \in C^{\infty}_{c}(\bbR \times \bbT; \bbR_{\geq 0})$.
\end{definition}
\begin{remark}[Adapted entropy condition via adapted entropy functional]
	Note, with the adapted entropy functional, the adapted entropy condition \eqref{eq:adaptedEntropyCondition_integralForm} can be formulated as: \\
	A function $\solution (\stochVar, \cdot, \cdot)$ is an adapted entropy solution, provided that for each $\alpha \in [\imageBound, \infty)$ (or $(-\infty, \imageBound]$) and the corresponding two steady state solutions $\stochSteadyStateSolutions$, it holds that $J^{\alpha}(\stochVar, \solution) \geq 0$, for every test function $\phi \in C^{\infty}_{c}(\bbR \times \bbT; \bbR_{\geq 0})$.
\end{remark}

Before we are able to prove the measurability of the solution map, we need to establish a few properties of the steady state equations $\stochSteadyStateSolutionArg{\pm}$ and the adapted entropy functional $J^{\alpha}$.
\begin{corollary}[Measurability of steady state solutions]
	\label{cor:measurability_steadyState}
	Let Assumptions \ref{as:stochFlux} and \ref{as:stochFlux_measurability} be satisfied. Then, for any constant $\steadyStatePar \in [\imageBound, \infty)$ (or $\steadyStatePar \in (-\infty, \imageBound]$), the stochastic steady state solutions $\stochSteadyStateSolution{\pm}$ defined by Equation \eqref{eq:stochSteadyState} are measurable in the sense that, for almost every $\spaceVar \in \bbR$, the mapping $\stochVar \mapsto \stochSteadyStateSolutions$ is measurable.
\end{corollary}

\begin{proof}
	By assumptions \ref{as:B-1}\--\ref{as:B-3}, we have the existence of the steady state solutions, which are measurable due to Assumption \ref{as:B-5}.
	Here, we mean measurability in the sense that, for each $\spaceVar \in \bbR$, the mapping $\stochVar \mapsto \stochSteadyStateSolutions$ is measurable.
\end{proof}
\begin{lemma}[Steady state solutions depend continuously on $\alpha$]
	\label{lem:steadyStateContinuity}
	Let Assumption \ref{as:stochFlux} be satisfied.
	Then, the steady state solutions defined via Equation \eqref{eq:stochSteadyState} depend continuously on $\alpha \in [\imageBound, \infty)$ (or $\alpha \in (-\infty, \imageBound]$).
\end{lemma}

\begin{proof}
	Let $\alpha, \beta \in [\imageBound, \infty)$ (or $\alpha, \beta \in (-\infty, \imageBound]$) with $\abs{\alpha - \beta} < \delta$.
	Consider the two steady state equations
	\begin{subequations}
		\begin{align}
			\fluxFunction(\stochVar, \spaceVar, \entropyConstant^{\pm}_{\alpha}) &= \alpha \quad \text{for a.e. } x \in \bbR \label{eq:steadyStateAlpha} , \\
			\fluxFunction(\stochVar, \spaceVar, \entropyConstant^{\pm}_{\beta}) &= \beta \quad \text{for a.e. } x \in \bbR \label{eq:steadyStateBeta} .
		\end{align}
	\end{subequations}
	By assumption, we know that there exist two steady state solutions $\entropyConstant^{\pm}_{\alpha}$ and $\entropyConstant^{\pm}_{\beta}$ for each \eqref{eq:steadyStateAlpha} and \eqref{eq:steadyStateBeta}, respectively.
	Thus, we may rewrite these equations as
	\begin{subequations}
		\begin{align}
			\entropyConstant^{\pm}_{\alpha} &= \invFlux(\stochVar, \spaceVar, \alpha) \quad \text{for a.e. } x \in \bbR \label{eq:steadyStateInvAlpha} , \\
			\entropyConstant^{\pm}_{\beta} &= \invFlux(\stochVar, \spaceVar, \beta) \quad \text{for a.e. } x \in \bbR \label{eq:steadyStateInvBeta} .
		\end{align}
	\end{subequations}
	Here $\invFlux(\stochVar, \spaceVar, \cdot)$ denotes the inverse of $\fluxFunction(\stochVar, \spaceVar, \cdot)$.
	Further, since $\fluxFunction(\stochVar, \spaceVar, \cdot)$ is a continuous map from $\bbR$ to $[\imageBound, \infty)$ (or $(-\infty, \imageBound]$), we get that $\invFlux(\stochVar, \spaceVar, \cdot)$ is a continuous function from either $[\imageBound, \infty)$ or $(-\infty, \imageBound]$ to $\bbR$.
	Now, consider
	\begin{equation}
		\begin{aligned}
			\abs{\entropyConstant^{\pm}_{\alpha} - \entropyConstant^{\pm}_{\beta}} = \abs{\invFlux(\stochVar, \spaceVar, \alpha) - \invFlux(\stochVar, \spaceVar, \beta)} 
			\leq C_{\stochVar} \abs{\alpha - \beta} ,
		\end{aligned}
	\end{equation}
	where we used the continuity of $\invFlux(\stochVar, \spaceVar, \cdot)$.
	Note, by Assumption \ref{as:B-6} the constant $C_{\stochVar}$ does not depend on $x \in \bbR$.
	Consequently, we have shown that the steady state solutions \eqref{eq:stochSteadyState} depend continuously on $\alpha$.
\end{proof}
\begin{proposition}[Adapted entropy functional depends continuously on $\alpha$]
	Let the flux function satisfy Assumptions \ref{as:stochFlux} and \ref{as:stochFlux_measurability}. 
	Let $\alpha \in [\imageBound, \infty)$ (or $(-\infty, \imageBound]$) be fixed and let $\stochSteadyStateSolution{\pm}$ denote the steady state solutions defined by Equation \eqref{eq:stochSteadyState}.
	Then the adapted entropy functional $J^{\alpha}$ defined in Equation \eqref{eq:entropyFunctional} depends continuously on $\alpha$.
\end{proposition}

\begin{proof}
	Let $\alpha, \beta \in [\imageBound, \infty)$ (or $(-\infty, \imageBound]$) with $\abs{\alpha - \beta} < \delta$.
	Due to Lemma \ref{lem:steadyStateContinuity}, it immediately follows that the terms \eqref{eq:entropyIntegral-1} and \eqref{eq:entropyIntegral-3} are continuous in $\alpha$.
	To show the continuity of \eqref{eq:entropyIntegral-2}, we define
	\begin{equation}
		\label{eq:G_defintion}
		G^{\alpha}(\omega, \nu) \coloneqq \int \sign (\nu(\spaceVar, \timeVar) - \stochSteadyStateSolutions) (\fluxFunction(\stochVar, \spaceVar, \nu(\spaceVar, \timeVar)) - \steadyStatePar) \spaceDeriv \phi(x, t) \d t \d x ,
	\end{equation}
	and consider
	\begin{equation}
		\label{eq:continuityDef_alpha}
		\begin{aligned}
			\abs{ G^{\alpha} (\omega, \nu(\spaceVar, \timeVar)) - G^{\beta}(\omega, \nu(\spaceVar, \timeVar))} &= \Big| \int \Big( \sign (\nu(\spaceVar, \timeVar) - m^{\pm}_{\alpha}(\stochVar, \spaceVar)) (\fluxFunction(\stochVar, \spaceVar, \nu(\spaceVar, \timeVar)) - \alpha) \\
			&\hspace*{2cm} - \sign(\nu(\spaceVar, \timeVar) - m^{\pm}_{\beta}(\stochVar, \spaceVar) (\fluxFunction(\stochVar, \spaceVar, \nu(\spaceVar, \timeVar)) - \beta ) \Big) \spaceDeriv \phi \d t \d x \Big| .
		\end{aligned}
	\end{equation}
	
	Define the following three sets:
	\begin{subequations}
		\begin{align}
			D^{=} &= \set{x \in \bbR : \sign(\nu(\spaceVar, \timeVar) - m^{\pm}_{\alpha}(\stochVar, \spaceVar)) = \sign(\nu(\spaceVar, \timeVar) - m^{\pm}_{\beta}(\stochVar, \spaceVar))} , \\
			D^{\pm} &= \set{x \in \bbR : \sign(\nu(\spaceVar, \timeVar) - m^{\pm}_{\alpha}(\stochVar, \spaceVar)) = - \sign(\nu(\spaceVar, \timeVar) - m^{\pm}_{\beta}(\stochVar, \spaceVar))} , \\
			D^{m} &= \set{x \in \bbR : \nu(\spaceVar, \timeVar) = m^{\pm}_{\alpha}(\stochVar, \spaceVar) \text{ or } \nu(\spaceVar, \timeVar) = m^{\pm}_{\beta}(\stochVar, \spaceVar)} .
		\end{align}
	\end{subequations}
	
	Note, these sets are disjoint, whenever we have $\alpha \neq \beta$, since we have 
	\begin{equation}
		D^{=} \cap D^{\pm} \cap D^{m} = \set{x \in \bbR: m^{\pm}_{\alpha}(\stochVar, \spaceVar) = \nu(\spaceVar, \timeVar) = m^{\pm}_{\beta}(\stochVar, \spaceVar)} .
	\end{equation}
	Hence, we can rewrite Equation \eqref{eq:continuityDef_alpha} as
	\begin{subequations}
		\label{eq:integralSplit_alpha}
		\begin{align}
			\Big| G^{\alpha} (\omega, \nu(\spaceVar, \timeVar)) &- G^{\beta}(\omega, \nu(\spaceVar, \timeVar)) \Big|  \\
			\begin{split} 
				&\leq \Big|\int_{D^{=}} \Big( \sign (\nu(\spaceVar, \timeVar) - m^{\pm}_{\alpha}(\stochVar, \spaceVar)) (\fluxFunction(\stochVar, \spaceVar, \nu(\spaceVar, \timeVar)) - \alpha) \\
				&\hspace*{3cm} - \sign(\nu(\spaceVar, \timeVar) - m^{\pm}_{\beta}(\stochVar, \spaceVar)) (\fluxFunction(\stochVar, \spaceVar, \nu(\spaceVar, \timeVar)) - \beta ) \Big) \spaceDeriv \phi \d t \d x \Big| \label{eq:integral_D=_alpha}
			\end{split}\\
			\begin{split}
				&\qquad + \Big| \int_{D^{\pm}} \Big( \sign (\nu(\spaceVar, \timeVar) - m^{\pm}_{\alpha}(\stochVar, \spaceVar)) (\fluxFunction(\stochVar, \spaceVar, \nu(\spaceVar, \timeVar)) - \alpha) \\
				&\hspace*{3cm} - \sign(\nu(\spaceVar, \timeVar) - m^{\pm}_{\beta}(\stochVar, \spaceVar)) (\fluxFunction(\stochVar, \spaceVar, \nu(\spaceVar, \timeVar)) - \beta ) \Big) \spaceDeriv \phi \d t \d x \Big| \label{eq:integral_Dpm_alpha}
			\end{split}\\
			\begin{split}
				&\qquad + \Big| \int_{D^{m}} \Big( \sign (\nu(\spaceVar, \timeVar) - m^{\pm}_{\alpha}(\stochVar, \spaceVar)) (\fluxFunction(\stochVar, \spaceVar, \nu(\spaceVar, \timeVar)) - \alpha) \\
				&\hspace*{3cm} - \sign(\nu(\spaceVar, \timeVar) - m^{\pm}_{\beta}(\stochVar, \spaceVar)) (\fluxFunction(\stochVar, \spaceVar, \nu(\spaceVar, \timeVar)) - \beta ) \Big) \spaceDeriv \phi \d t \d x \Big| \label{eq:integral_Dm_alpha} .
			\end{split}
		\end{align}
	\end{subequations}
	
	We will now estimate the terms in \eqref{eq:integralSplit_alpha} separately, starting with \eqref{eq:integral_D=_alpha}:
	\begin{equation}
		\begin{aligned}
			\Big| \int_{D^{=}} \Big( \sign (\nu(\spaceVar, \timeVar) &- m^{\pm}_{\alpha}(\stochVar, \spaceVar)) (\fluxFunction(\stochVar, \spaceVar, \nu(\spaceVar, \timeVar)) - \alpha) \\
			&- \sign(\nu(\spaceVar, \timeVar) - m^{\pm}_{\beta}(\stochVar, \spaceVar)) (\fluxFunction(\stochVar, \spaceVar, \nu(\spaceVar, \timeVar)) - \beta ) \Big) \spaceDeriv \phi \d t \d x \Big| \\
			&\hspace*{2cm}\leq \int_{D^{=}} \Big| \sign (\nu(\spaceVar, \timeVar) - m^{\pm}_{\alpha}(\stochVar, \spaceVar)) (\fluxFunction(\stochVar, \spaceVar, \nu(\spaceVar, \timeVar)) - \alpha) \\
			&\hspace*{4cm} - \sign(\nu(\spaceVar, \timeVar) - m^{\pm}_{\beta}(\stochVar, \spaceVar)) (\fluxFunction(\stochVar, \spaceVar, \nu(\spaceVar, \timeVar)) - \beta ) \spaceDeriv \phi \Big| \d t \d x \\
			&\hspace*{2cm}\leq \int_{D^{=}} \abs{ \Big( \fluxFunction(\stochVar, \spaceVar, \nu(\spaceVar, \timeVar)) - \alpha - \fluxFunction(\stochVar, \spaceVar, \nu(\spaceVar, \timeVar)) + \beta \Big) \spaceDeriv \phi} \d t \d x \\
			&\hspace*{2cm}\leq C_{\phi} \abs{\alpha - \beta} .
		\end{aligned}
	\end{equation}
	Here, we first used the definition of $D^{=}$ and afterwards the boundedness of $\spaceDeriv \phi$.
	
	Now, consider the integral \eqref{eq:integral_Dpm_alpha}:
	\begin{equation}
		\begin{aligned}
			\Big| \int_{D^{\pm}} \Big( \sign (&\nu(\spaceVar, \timeVar) - m^{\pm}_{\alpha}(\stochVar, \spaceVar)) (\fluxFunction(\stochVar, \spaceVar, \nu(\spaceVar, \timeVar)) - \alpha) \\
			&- \sign(\nu(\spaceVar, \timeVar) - m^{\pm}_{\beta}(\stochVar, \spaceVar) (\fluxFunction(\stochVar, \spaceVar, \nu(\spaceVar, \timeVar)) - \beta ) \Big) \spaceDeriv \phi \d t \d x \Big| \\
			&\hspace*{1.5cm}\leq \int_{D^{\pm}} \Big| \sign (\nu(\spaceVar, \timeVar) - m^{\pm}_{\alpha}(\stochVar, \spaceVar)) (\fluxFunction(\stochVar, \spaceVar, \nu(\spaceVar, \timeVar)) - \alpha) \\
			&\hspace*{3cm} - \sign(\nu(\spaceVar, \timeVar) - m^{\pm}_{\beta}(\stochVar, \spaceVar)) (\fluxFunction(\stochVar, \spaceVar, \nu(\spaceVar, \timeVar)) - \beta ) \spaceDeriv \phi \Big| \d t \d x \\
			&\hspace*{1.5cm}\leq \int_{D^{\pm}} \abs{ (\fluxFunction(\stochVar, \spaceVar, \nu(\spaceVar, \timeVar)) - \alpha) \spaceDeriv \phi } \d t \d x + \int_{D^{\pm}} \abs{ (\fluxFunction(\stochVar, \spaceVar, \nu(\spaceVar, \timeVar)) - \beta) \spaceDeriv \phi } \d t \d x \\
			&\hspace*{1.5cm}\leq \int_{D^{\pm}} L_{I} \abs{ \nu(\spaceVar, \timeVar) - m^{\pm}_{\alpha}(\stochVar, \spaceVar)} \abs{ \spaceDeriv \phi } \d t \d x + \int_{D^{\pm}} L_{I} \abs{ \nu(\spaceVar, \timeVar) - m^{\pm}_{\beta}(\stochVar, \spaceVar)} \abs{ \spaceDeriv \phi } \d t \d x .
		\end{aligned}
	\end{equation}
	In the second step, we used the definition of $D^{\pm}$ together with the triangle inequality and estimated $\sign(\nu - m^{\pm}_{\alpha}) \leq 1$.
	Afterwards, we inserted the definition of $\alpha$ and $\beta$ via Equations \eqref{eq:steadyStateAlpha} and \eqref{eq:steadyStateBeta}, respectively, and used the local Lipschitz continuity of $\fluxFunction$.
	Using now the fact that for $\xi_1, \xi_2 \in \bbR$, we have the identity $\abs{\xi_1 - \xi_2} = \sign(\xi_1 - \xi_2)(\xi_1 - \xi_2)$, we obtain
	\begin{equation}
		\begin{aligned}
			\Big| \int_{D^{\pm}} \Big( \sign (&\nu(\spaceVar, \timeVar) - m^{\pm}_{\alpha}(\stochVar, \spaceVar)) (\fluxFunction(\stochVar, \spaceVar, \nu(\spaceVar, \timeVar)) - \alpha) \\
			&- \sign(\nu(\spaceVar, \timeVar) - m^{\pm}_{\beta}(\stochVar, \spaceVar) (\fluxFunction(\stochVar, \spaceVar, \nu(\spaceVar, \timeVar)) - \beta ) \Big) \spaceDeriv \phi \d t \d x \Big| \\
			&\hspace*{0.5cm}\leq \int_{D^{\pm}} L_{I} \Big( \sign(\nu(\spaceVar, \timeVar) - m^{\pm}_{\alpha}(\stochVar, \spaceVar)) (\nu(\spaceVar, \timeVar) - m^{\pm}_{\alpha}(\stochVar, \spaceVar)) \\
			&\hspace*{3cm} + \sign(\nu(\spaceVar, \timeVar) - m^{\pm}_{\beta}(\stochVar, \spaceVar)) (\nu(\spaceVar, \timeVar) - m^{\pm}_{\beta}(\stochVar, \spaceVar)) \Big) \abs{ \spaceDeriv \phi } \d t \d x \\
			&\hspace*{0.5cm}= \int_{D^{\pm}} L_{I} \sign(\nu(\spaceVar, \timeVar) - m^{\pm}_{\alpha}(\stochVar, \spaceVar)) \Big( \nu(\spaceVar, \timeVar) - m^{\pm}_{\alpha}(\stochVar, \spaceVar) - \nu(\spaceVar, \timeVar) + m^{\pm}_{\beta}(\stochVar, \spaceVar) \Big) \abs{ \spaceDeriv \phi } \d t \d x \\
			&\hspace*{0.5cm}\leq \int_{D^{\pm}} L_{I} \abs{ m^{\pm}_{\alpha}(\stochVar, \spaceVar) - m^{\pm}_{\beta}(\stochVar, \spaceVar)} \abs{\spaceDeriv \phi} \d t \d x \\
			&\hspace*{0.5cm}\leq L_{I} C_{\phi} \abs{ m^{\pm}_{\alpha}(\stochVar, \spaceVar) - m^{\pm}_{\beta}(\stochVar, \spaceVar)} \\
			&\hspace*{0.5cm}\leq L_{I} C_{\phi} C_{\omega} \abs{\alpha - \beta}.
		\end{aligned}
	\end{equation}
	Here, we used in the last step the continuity of $m^{\pm}_{\alpha}$ in $\alpha$ from Lemma \ref{lem:steadyStateContinuity}.
	It remains to estimate \eqref{eq:integral_Dm_alpha}.
	Therefore, we define the two sets
	\begin{subequations}
		\begin{align}
			D^{m}_{\alpha} &= \set{x \in \bbR: \nu(\spaceVar, \timeVar) = m^{\pm}_{\alpha}(\stochVar, \spaceVar)} , \\
			D^{m}_{\beta} &= \set{x \in \bbR: \nu(\spaceVar, \timeVar) = m^{\pm}_{\beta}(\stochVar, \spaceVar)} .
		\end{align}
	\end{subequations}
	Note that these sets are disjoint, whenever $\alpha \neq \beta$.
	With the definition of $D^{m}_{\alpha}$ and $D^{m}_{\beta}$ and the triangle inequality, we obtain for \eqref{eq:integral_Dm_alpha}
	\begin{equation}
		\begin{aligned}
			\Big| \int_{D^{m}} \Big( \sign (\nu(\spaceVar, \timeVar) &- m^{\pm}_{\alpha}(\stochVar, \spaceVar)) (\fluxFunction(\stochVar, \spaceVar, \nu(\spaceVar, \timeVar)) - \alpha) \\
			&- \sign(\nu(\spaceVar, \timeVar) - m^{\pm}_{\beta}(\stochVar, \spaceVar)) (\fluxFunction(\stochVar, \spaceVar, \nu(\spaceVar, \timeVar)) - \beta ) \Big) \spaceDeriv \phi \d t \d x \Big| \\
			&\hspace*{1.5cm}\leq \Big| \int_{D^{m}_{\alpha}} - \sign(\nu(\spaceVar, \timeVar) - m^{\pm}_{\beta}(\stochVar, \spaceVar)) (\fluxFunction(\stochVar, \spaceVar, \nu(\spaceVar, \timeVar)) - \beta ) \spaceDeriv \phi \d t \d x \Big| \\
			&\hspace*{3cm} + \Big| \int_{D^{m}_{\beta}} \sign (\nu(\spaceVar, \timeVar) - m^{\pm}_{\alpha}(\stochVar, \spaceVar)) (\fluxFunction(\stochVar, \spaceVar, \nu(\spaceVar, \timeVar)) - \alpha) \spaceDeriv \phi \d t \d x \Big| \\
			&\hspace*{1.5cm}\leq \int_{D^{m}_{\alpha}} \abs{\fluxFunction(\stochVar, \spaceVar, \nu(\spaceVar, \timeVar)) - \beta} \abs{\spaceDeriv \phi} \d t \d x + \int_{D^{m}_{\beta}} \abs{\fluxFunction(\stochVar, \spaceVar, \nu(\spaceVar, \timeVar)) - \alpha} \abs{\spaceDeriv \phi} \d t \d x . \\
		\end{aligned}
	\end{equation}
	Now, using the steady state equations \eqref{eq:steadyStateAlpha} and \eqref{eq:steadyStateBeta} and the local Lipschitz continuity of $\fluxFunction$, we get
	\begin{equation}
		\begin{aligned}
			\Big| \int_{D^{m}} \Big( &\sign (\nu(\spaceVar, \timeVar) - m^{\pm}_{\alpha}(\stochVar, \spaceVar)) (\fluxFunction(\stochVar, \spaceVar, \nu(\spaceVar, \timeVar)) - \alpha) \\
			&- \sign(\nu(\spaceVar, \timeVar) - m^{\pm}_{\beta}(\stochVar, \spaceVar)) (\fluxFunction(\stochVar, \spaceVar, \nu(\spaceVar, \timeVar)) - \beta ) \Big) \spaceDeriv \phi \d t \d x \Big| \\
			&\hspace*{1.5cm}= \int_{D^{m}_{\alpha}} \abs{\fluxFunction(\stochVar, \spaceVar, \nu(\spaceVar, \timeVar)) - \fluxFunction(\stochVar, \spaceVar, m^{\pm}_{\beta}(\stochVar, \spaceVar))} \abs{\spaceDeriv \phi} \d t \d x \\
			&\hspace*{3cm}+ \int_{D^{m}_{\beta}} \abs{\fluxFunction(\stochVar, \spaceVar, \nu(\spaceVar, \timeVar)) - \fluxFunction(\stochVar, \spaceVar, m^{\pm}_{\alpha}(\stochVar, \spaceVar))} \abs{\spaceDeriv \phi} \d t \d x \\
			&\hspace*{1.5cm}= \int_{D^{m}_{\alpha}} \abs{\fluxFunction(\stochVar, \spaceVar, m^{\pm}_{\alpha}(\stochVar, \spaceVar)) - \fluxFunction(\stochVar, \spaceVar, m^{\pm}_{\beta}(\stochVar, \spaceVar))} \abs{\spaceDeriv \phi} \d t \d x \\
			&\hspace*{3cm}+ \int_{D^{m}_{\beta}} \abs{\fluxFunction(\stochVar, \spaceVar, m^{\pm}_{\beta}(\stochVar, \spaceVar)) - \fluxFunction(\stochVar, \spaceVar, m^{\pm}_{\alpha}(\stochVar, \spaceVar))} \abs{\spaceDeriv \phi} \d t \d x \\
			&\hspace*{1.5cm}\leq \int_{D^{m}_{\alpha}} L_{I} \abs{m^{\pm}_{\alpha}(\stochVar, \spaceVar) - m^{\pm}_{\beta}(\stochVar, \spaceVar)} \abs{\spaceDeriv \phi} \d t \d x + \int_{D^{m}_{\beta}} L_{I} \abs{m^{\pm}_{\beta}(\stochVar, \spaceVar) - m^{\pm}_{\alpha}(\stochVar, \spaceVar)} \abs{\spaceDeriv \phi} \d t \d x \\
			&\hspace*{1.5cm}\leq 2 L_{I} C_{\phi} C_{\omega} \abs{\alpha - \beta} .
		\end{aligned}
	\end{equation}
	Here, we used Lemma \ref{lem:steadyStateContinuity} in the last step.
	Consequently, we get the following estimation for \eqref{eq:integralSplit_alpha}:
	\begin{equation}
		\begin{aligned}
			\Big| G^{\alpha} (\omega, \nu) - G^{\beta}(\omega, \nu) \Big| &\leq C_{\phi} \abs{\alpha - \beta} + L_{I} C_{\phi} C_{\omega} \abs{\alpha - \beta} + 2 L_{I} C_{\phi} C_{\omega} \abs{\alpha - \beta} \\
			&\leq \Big( C_{\phi} + 3 L_{I} C_{\phi} C_{\omega} \Big) \abs{\alpha - \beta}
		\end{aligned}
	\end{equation}
	and thus, the functional $J^{\alpha}$ depends continuously on $\alpha$.
\end{proof}

As a last property of the adapted entropy functional, we will show that $J^{\alpha}$ is a Carathéodory functional.
\begin{proposition}[Adapted entropy functional is Carathéodory]
	\label{prop:EntropyFunctional_Caratheodory}
	Let the flux function $\fluxFunction$ satisfy Assumptions \ref{as:stochFlux} and \ref{as:stochFlux_measurability} and let $\initialCond \in L^{p}(\Omega, L^{\infty}(\bbR))$.
	Further, let $\alpha \in [\imageBound, \infty)$ (or $(-\infty, \imageBound]$) be fixed and let $\stochSteadyStateSolution{\pm}$ denote the steady state solutions defined by Equation \eqref{eq:stochSteadyState}.
	Then the adapted entropy functional $J^{\alpha}$, defined in \eqref{eq:entropyFunctional} is a Carathéodory map, i.e., $J^{\alpha}(\omega, \nu)$ is measurable in $\stochVar \in \Omega$ and continuous in $\nu \in \cX$.
\end{proposition}

\begin{proof}
	We divide the step into two steps: we first show the continuity with respect to $\nu \in \cX$ and afterwards the measurability in $\omega \in \Omega$.
	
	\textit{Step 1:}
	We start by showing the continuity of $J^{\alpha}$ w.r.t. $\nu \in \cX$.
	The continuity of the first integral \eqref{eq:entropyIntegral-1} and the third integral \eqref{eq:entropyIntegral-3} is obvious.
	For the second integral \eqref{eq:entropyIntegral-2}, we explicitly show the continuity: \\
	Let $\stochVar \in \stochDomain$ and $\alpha \in [\imageBound, \infty)$ (or $(-\infty, \imageBound]$) be fixed and let $\nu, \tilde{\nu} \in \cX$ with $\norm{\nu - \tilde{\nu}}_{\cX} < \delta$, for some $\delta > 0$. 
	Consider now
	\begin{equation}
		\begin{aligned}
			\abs{G^{\alpha}(\omega, \nu) - G^{\alpha}(\omega, \tilde{\nu})} &= \Big| \int \Big(\sign (\nu(\spaceVar, \timeVar) - \stochSteadyStateSolutions) (\fluxFunction(\stochVar, \spaceVar, \nu(\spaceVar, \timeVar)) - \steadyStatePar) \\
			&\qquad \qquad \qquad - \sign (\tilde{\nu}(\spaceVar, \timeVar) - \stochSteadyStateSolutions) (\fluxFunction(\stochVar, \spaceVar, \tilde{\nu}(\spaceVar, \timeVar)) - \steadyStatePar) \Big) \spaceDeriv \phi(x, t) \d t \d x \Big| ,
		\end{aligned}
	\end{equation}
	where $G^{\alpha}(\omega, \nu)$ is defined as in Equation \eqref{eq:G_defintion}.
	Let us define the following three sets:
	\begin{equation}
		\begin{aligned}
			D^{=} &\coloneqq \set{x \in \bbR | \sign (\nu(\spaceVar, \timeVar) - \stochSteadyStateSolutions) = \sign (\tilde{\nu}(\spaceVar, \timeVar) - \stochSteadyStateSolutions)} , \\
			D^{\entropyConstant} &\coloneqq \set{x \in \bbR | \nu(\spaceVar, \timeVar) = \stochSteadyStateSolutions \text{ or } \tilde{\nu}(\spaceVar, \timeVar) = \stochSteadyStateSolutions} , \\
			D^{\pm} &\coloneqq \set{x \in \bbR | \sign (\nu(\spaceVar, \timeVar) - \stochSteadyStateSolutions) = - \sign (\tilde{\nu}(\spaceVar, \timeVar) - \stochSteadyStateSolutions)} .
		\end{aligned}
	\end{equation}
	
	Note that in general we have $D^{=} \cap D^{\entropyConstant} \cap D^{\pm} \neq \emptyset$, but we have $\bbR = D^{=} \cup D^{\entropyConstant} \cup D^{\pm}$. 
	Thus, we can write
	\begin{subequations}
		\label{eq:continuityEstimate}
		\begin{align}
			\abs{G^{\alpha}(\omega, \nu) - G^{\alpha}(\omega, \tilde{\nu})} &= \Big| \int \Big(\sign (\nu(\spaceVar, \timeVar) - \stochSteadyStateSolutions) (\fluxFunction(\stochVar, \spaceVar, \nu(\spaceVar, \timeVar)) - \steadyStatePar) \\
			&\qquad \qquad \qquad - \sign (\tilde{\nu}(\spaceVar, \timeVar) - \stochSteadyStateSolutions) (\fluxFunction(\stochVar, \spaceVar, \tilde{\nu}(\spaceVar, \timeVar)) - \steadyStatePar) \Big) \spaceDeriv \phi(x, t) \d t \d x \Big| \\
			\begin{split}
				&\leq \int_{D^{=} \times \bbT} \Big| \sign (\nu(\spaceVar, \timeVar) - \stochSteadyStateSolutions) (\fluxFunction(\stochVar, \spaceVar, \nu(\spaceVar, \timeVar)) - \steadyStatePar) \label{eq:integral_D=} \\
				&\qquad \qquad \qquad - \sign (\tilde{\nu}(\spaceVar, \timeVar) - \stochSteadyStateSolutions) (\fluxFunction(\stochVar, \spaceVar, \tilde{\nu}(\spaceVar, \timeVar)) - \steadyStatePar) \Big| \abs{\spaceDeriv \phi(x, t)} \d t \d x 
			\end{split}\\
			\begin{split}
				&\qquad + \int_{D^{\entropyConstant} \times \bbT} \Big| \sign (\nu(\spaceVar, \timeVar) - \stochSteadyStateSolutions) (\fluxFunction(\stochVar, \spaceVar, \nu(\spaceVar, \timeVar)) - \steadyStatePar) \label{eq:integral_Dm} \\
				&\qquad \qquad \qquad - \sign (\tilde{\nu}(\spaceVar, \timeVar) - \stochSteadyStateSolutions) (\fluxFunction(\stochVar, \spaceVar, \tilde{\nu}(\spaceVar, \timeVar)) - \steadyStatePar) \Big| \abs{\spaceDeriv \phi(x, t)} \d t \d x
			\end{split}\\
			\begin{split}
				&\qquad + \int_{D^{\pm} \times \bbT} \Big| \sign (\nu(\spaceVar, \timeVar) - \stochSteadyStateSolutions) (\fluxFunction(\stochVar, \spaceVar, \nu(\spaceVar, \timeVar)) - \steadyStatePar) \label{eq:integral_Dpm} \\
				&\qquad \qquad \qquad - \sign (\tilde{\nu}(\spaceVar, \timeVar) - \stochSteadyStateSolutions) (\fluxFunction(\stochVar, \spaceVar, \tilde{\nu}(\spaceVar, \timeVar)) - \steadyStatePar) \Big| \abs{\spaceDeriv \phi(x, t)} \d t \d x.
			\end{split}
		\end{align}
	\end{subequations}
	
	When considering the integral \eqref{eq:integral_D=}, we obtain
	\begin{equation}
		\begin{aligned}
			\int\limits_{D^{=} \times \bbT} \Big| \Big( \sign &(\nu(\spaceVar, \timeVar) - \stochSteadyStateSolutions) (\fluxFunction(\stochVar, \spaceVar, \nu(\spaceVar, \timeVar)) - \steadyStatePar) \\
			& - \sign (\tilde{\nu}(\spaceVar, \timeVar) - \stochSteadyStateSolutions) (\fluxFunction(\stochVar, \spaceVar, \tilde{\nu}(\spaceVar, \timeVar)) - \steadyStatePar) \Big) \spaceDeriv \phi(x, t) \Big| \d t \d x \\[5pt]
			&\hspace*{3cm} = \int\limits_{D^{=} \times \bbT} \Big| \Big( (\fluxFunction(\stochVar, \spaceVar, \nu(\spaceVar, \timeVar)) - \steadyStatePar) - (\fluxFunction(\stochVar, \spaceVar, \tilde{\nu}(\spaceVar, \timeVar)) - \steadyStatePar) \Big) \spaceDeriv \phi(x, t) \Big| \d t \d x \\
			&\hspace*{3cm} = \int\limits_{D^{=} \times \bbT} \Big| \fluxFunction(\stochVar, \spaceVar, \nu(\spaceVar, \timeVar)) - \fluxFunction(\stochVar, \spaceVar, \tilde{\nu}(\spaceVar, \timeVar)) \Big| \abs{\spaceDeriv \phi(x, t)} \d t \d x \\
			&\hspace*{3cm} \leq \int\limits_{D^{=} \times \bbT} L_I \abs{\nu(\spaceVar, \timeVar) - \tilde{\nu}(\spaceVar, \timeVar)} \abs{\spaceDeriv \phi(x, t)} \d t \d x \\
			&\hspace*{3cm} \leq C_{\phi} L_I \norm{\nu - \tilde{\nu}}_{L^{\infty}(D^{=} \times \bbT)} .
		\end{aligned}
	\end{equation}
	Here, we first used the definition of $D^{=}$ to estimate the sign functions.
	Afterwards, we used the locally Lipschitz continuity of $\fluxFunction$ in the third argument by Assumption \ref{as:B-3} or \ref{as:B-3alt}.
	Note that the Lipschitz constant depends on the interval $I \subset \bbR$, which contains $\nu(x,t), \tilde{\nu}(x,t)$ for all $(x,t) \in \bbR \times \bbT$. 
	This interval exists, since by hypothesis we have $\nu, \tilde{\nu} \in \cX$, which implies $\nu, \tilde{\nu} \in L^{\infty}(\bbR \times \bbT)$.
	
	Now, before we consider the integral \eqref{eq:integral_Dm}, we define
	\begin{equation}
		\label{eq:Dm_splitting}
		\begin{aligned}
			D^{\entropyConstant}_{\nu} &\coloneqq \set{x \in D^{\entropyConstant} | \nu(\spaceVar, \timeVar) = \stochSteadyStateSolutions} \\
			D^{\entropyConstant}_{\tilde{\nu}} &\coloneqq \set{x \in D^{\entropyConstant} | \tilde{\nu}(\spaceVar, \timeVar) = \stochSteadyStateSolutions} .
		\end{aligned}
	\end{equation}
	With this, we can rewrite the integral \eqref{eq:integral_Dm} as
	\begin{equation}
		\begin{aligned}
			\int\limits_{D^{\entropyConstant} \times \bbT} \Big| \sign (&\nu(\spaceVar, \timeVar) - \stochSteadyStateSolutions) (\fluxFunction(\stochVar, \spaceVar, \nu(\spaceVar, \timeVar)) - \steadyStatePar) \\
			&- \sign (\tilde{\nu}(\spaceVar, \timeVar) - \stochSteadyStateSolutions) (\fluxFunction(\stochVar, \spaceVar, \tilde{\nu}(\spaceVar, \timeVar)) - \steadyStatePar) \Big| \abs{\spaceDeriv \phi(x, t)} \d t \d x \\
			&\hspace*{1cm} = \int\limits_{D^{\entropyConstant}_{\nu} \times \bbT} \Big| \sign (\tilde{\nu}(\spaceVar, \timeVar) - \stochSteadyStateSolutions) (\fluxFunction(\stochVar, \spaceVar, \tilde{\nu}(\spaceVar, \timeVar)) - \steadyStatePar) \Big| \abs{\spaceDeriv \phi(x, t)} \d t \d x \\
			&\hspace*{3cm} + \int\limits_{D^{\entropyConstant}_{\tilde{\nu}} \times \bbT} \Big| \sign (\nu(\spaceVar, \timeVar) - \stochSteadyStateSolutions) (\fluxFunction(\stochVar, \spaceVar, \nu(\spaceVar, \timeVar)) - \steadyStatePar) \Big| \abs{\spaceDeriv \phi(x, t)} \d t \d x \\
			&\hspace*{1cm} \leq \int\limits_{D^{\entropyConstant}_{\nu} \times \bbT} \Big|  \fluxFunction(\stochVar, \spaceVar, \tilde{\nu}(\spaceVar, \timeVar)) - \steadyStatePar \Big| \abs{ \spaceDeriv \phi(x,t)} \d t \d x + \int\limits_{D^{\entropyConstant}_{\tilde{\nu}} \times \bbT} \Big| \fluxFunction(\stochVar, \spaceVar, \nu(\spaceVar, \timeVar)) - \steadyStatePar \Big| \abs{\spaceDeriv \phi(x, t)} \d t \d x .
		\end{aligned}
	\end{equation}
	Now, note that Equation \eqref{eq:Dm_splitting} gives by definition
	\begin{equation}
		\begin{aligned}
			\alpha &= \fluxFunction(\omega, x, \stochSteadyStateSolutions) = \fluxFunction(\stochVar, \spaceVar, \nu(\spaceVar, \timeVar)) \qquad \text{for } x \in D^{\entropyConstant}_{\nu} , \\
			\alpha &= \fluxFunction(\omega, x, \stochSteadyStateSolutions) = \fluxFunction(\stochVar, \spaceVar, \tilde{\nu}(\spaceVar, \timeVar)) \qquad \text{for } x \in D^{\entropyConstant}_{\tilde{\nu}} .
		\end{aligned}
	\end{equation}
	Consequently, we obtain
	\begin{equation}
		\begin{aligned}
			\int\limits_{D^{\entropyConstant} \times \bbT} \Big| \sign (&\nu(\spaceVar, \timeVar) - \stochSteadyStateSolutions) (\fluxFunction(\stochVar, \spaceVar, \nu(\spaceVar, \timeVar)) - \steadyStatePar) \\
			&- \sign (\tilde{\nu}(\spaceVar, \timeVar) - \stochSteadyStateSolutions) (\fluxFunction(\stochVar, \spaceVar, \tilde{\nu}(\spaceVar, \timeVar)) - \steadyStatePar) \Big| \abs{\spaceDeriv \phi(x, t)} \d t \d x \\
			&\hspace*{1cm} \leq \int\limits_{D^{\entropyConstant}_{\nu} \times \bbT} \Big|  \fluxFunction(\stochVar, \spaceVar, \tilde{\nu}(\spaceVar, \timeVar)) - \steadyStatePar \Big| \abs{ \spaceDeriv \phi(x,t)} \d t \d x + \int\limits_{D^{\entropyConstant}_{\tilde{\nu}} \times \bbT} \Big| \fluxFunction(\stochVar, \spaceVar, \nu(\spaceVar, \timeVar)) - \steadyStatePar \Big| \abs{\spaceDeriv \phi(x, t)} \d t \d x \\
			&\hspace*{1cm} = \int_{D^{\entropyConstant}_{\nu} \times \bbT} \Big|  \fluxFunction(\stochVar, \spaceVar, \tilde{\nu}(\spaceVar, \timeVar)) - \fluxFunction(\stochVar, \spaceVar, \nu(\spaceVar, \timeVar)) \Big| \abs{\spaceDeriv \phi(x,t)} \d t \d x \\
			&\hspace*{3cm} + \int_{D^{\entropyConstant}_{\tilde{\nu}} \times \bbT} \Big| \fluxFunction(\stochVar, \spaceVar, \nu(\spaceVar, \timeVar)) - \fluxFunction(\stochVar, \spaceVar, \tilde{\nu}(\spaceVar, \timeVar)) \Big| \abs{\spaceDeriv \phi(x,t)} \d t \d x \\
			&\hspace*{1cm} \leq C_{\phi} L_I \norm{\nu - \tilde{\nu}}_{L^{\infty}(D^{\entropyConstant}_{\nu} \times \bbT)}  + C_{\phi} L_I \norm{\nu - \tilde{\nu}}_{L^{\infty}(D^{\entropyConstant}_{\tilde{\nu}} \times \bbT)} \\
			&\hspace*{1cm} \leq 2 C_{\phi} L_I \norm{\nu - \tilde{\nu}}_{L^{\infty}(D^{\entropyConstant} \times \bbT)} .
		\end{aligned}
	\end{equation}
	It remains to consider the integral \eqref{eq:integral_Dpm}:
	\begin{equation}
		\begin{aligned}
			\int_{D^{\pm} \times \bbT} \Big| \sign (\nu(\spaceVar, \timeVar) - \stochSteadyStateSolutions) (\fluxFunction(&\stochVar, \spaceVar, \nu(\spaceVar, \timeVar)) - \steadyStatePar) \\
			&- \sign (\tilde{\nu}(\spaceVar, \timeVar) - \stochSteadyStateSolutions) (\fluxFunction(\stochVar, \spaceVar, \tilde{\nu}(\spaceVar, \timeVar)) - \steadyStatePar) \Big| \abs{\spaceDeriv \phi(x,t)} \d t \d x.
		\end{aligned}
	\end{equation}
	We start with the estimation
	\begin{equation}
		\begin{aligned}
			\int\limits_{D^{\pm} \times \bbT} \Big| \sign (&\nu(\spaceVar, \timeVar) - \stochSteadyStateSolutions) (\fluxFunction(\stochVar, \spaceVar, \nu(\spaceVar, \timeVar)) - \steadyStatePar) \\
			&- \sign (\tilde{\nu}(\spaceVar, \timeVar) - \stochSteadyStateSolutions) (\fluxFunction(\stochVar, \spaceVar, \tilde{\nu}(\spaceVar, \timeVar)) - \steadyStatePar) \Big| \abs{\spaceDeriv \phi(x,t)} \d t \d x \\
			&\hspace*{2cm} \leq \int\limits_{D^{\pm} \times \bbT} \Big| \fluxFunction(\stochVar, \spaceVar, \nu(\spaceVar, \timeVar)) - \steadyStatePar \Big| \abs{\spaceDeriv \phi(x,t)} \d t \d x \\
			&\hspace*{4cm} + \int\limits_{D^{\pm} \times \bbT} \Big| \fluxFunction(\stochVar, \spaceVar, \tilde{\nu}(\spaceVar, \timeVar)) - \steadyStatePar \Big| \abs{\spaceDeriv \phi(x,t)} \d t \d x \\
			&\hspace*{2cm} = \int\limits_{D^{\pm} \times \bbT} \Big| \fluxFunction(\stochVar, \spaceVar, \nu(\spaceVar, \timeVar)) - \fluxFunction(\stochVar, \spaceVar, \stochSteadyStateSolutions) \Big| \abs{\spaceDeriv \phi(x,t)} \d t \d x \\
			&\hspace*{4cm} + \int\limits_{D^{\pm} \times \bbT} \Big| \fluxFunction(\stochVar, \spaceVar, \tilde{\nu}(\spaceVar, \timeVar)) - \fluxFunction(\stochVar, \spaceVar, \stochSteadyStateSolutions) \Big| \abs{\spaceDeriv \phi(x,t)} \d t \d x \\ 
			&\hspace*{2cm} \leq \int\limits_{D^{\pm} \times \bbT} L_I \abs{\nu(\spaceVar, \timeVar) - \stochSteadyStateSolutions} \abs{\spaceDeriv \phi(x,t)} \d t \d x \\
			&\hspace*{4cm} + \int\limits_{D^{\pm} \times \bbT} L_I \abs{\tilde{\nu}(\spaceVar, \timeVar) - \stochSteadyStateSolutions} \abs{\spaceDeriv \phi(x,t)} \d t \d x \\
			&\hspace*{2cm} \leq \int\limits_{D^{\pm} \times \bbT} L_I \Big( \abs{\nu(\spaceVar, \timeVar) - \stochSteadyStateSolutions} + \abs{\tilde{\nu}(\spaceVar, \timeVar) - \stochSteadyStateSolutions} \Big) \abs{\spaceDeriv \phi(x,t)} \d t \d x . \\
		\end{aligned}
	\end{equation}
	Here, we first used the steady state equation \eqref{eq:stochSteadyState} and then the local Lipschitz continuity of $\fluxFunction$ in the third argument.
	Note that the Lipschitz constant is dependent on the interval $I$ containing $\nu(x,t), \tilde{\nu}(x,t)$ and $\stochSteadyStateSolutions$ for all $(x,t) \in \bbR \times \bbT$.
	The interval exists, since $\nu, \tilde{\nu} \in L^{\infty}(\bbR \times \bbT)$ by hypothesis and, for each $\omega, \alpha$, we have $\stochSteadyStateSolutions \in L^{\infty}(\bbR \times \bbT)$ due to Assumption \ref{as:B-2}.
	Note, even though we do not explicitly know that the above integrals are well defined, their existence is ensured via dominated convergence with the subsequent estimation.
	
	By the definition of $D^{\pm}$, we have 
	\begin{equation}
		\label{eq:D_plusminus}
		\sign (\nu(\spaceVar, \timeVar) - \stochSteadyStateSolutions) = - \sign (\tilde{\nu}(\spaceVar, \timeVar) - \stochSteadyStateSolutions) \qquad \text{for } x \in D^{\pm} .
	\end{equation}
	
	For $\xi_1, \xi_2 \in \bbR$ it holds that $\abs{\xi_1 - \xi_2} = \sign(\xi_1 - \xi_2)(\xi_1 - \xi_2)$.
	Together with \eqref{eq:D_plusminus} this yields
	\begin{equation}
		\begin{aligned}
			\int\limits_{D^{\pm} \times \bbT} \Big| \sign (&\nu(\spaceVar, \timeVar) - \stochSteadyStateSolutions) (\fluxFunction(\stochVar, \spaceVar, \nu(\spaceVar, \timeVar)) - \steadyStatePar) \\
			&- \sign (\tilde{\nu}(\spaceVar, \timeVar) - \stochSteadyStateSolutions) (\fluxFunction(\stochVar, \spaceVar, \tilde{\nu}(\spaceVar, \timeVar)) - \steadyStatePar) \Big| \abs{\spaceDeriv \phi(x,t)} \d t \d x \\
			&\hspace*{2cm} \leq \int\limits_{D^{\pm} \times \bbT} L_I \Big( \abs{\nu(\spaceVar, \timeVar) - \stochSteadyStateSolutions} + \abs{\tilde{\nu}(\spaceVar, \timeVar) - \stochSteadyStateSolutions} \Big) \abs{\spaceDeriv \phi(x,t)} \d t \d x \\
			&\hspace*{2cm} = \int\limits_{D^{\pm} \times \bbT} L_I \Big( \sign(\nu(\spaceVar, \timeVar) - \stochSteadyStateSolutions)(\nu(\spaceVar, \timeVar) - \stochSteadyStateSolutions) \\
			&\hspace*{4cm} + \sign(\tilde{\nu}(\spaceVar, \timeVar) - \stochSteadyStateSolutions)(\tilde{\nu}(\spaceVar, \timeVar) - \stochSteadyStateSolutions) \Big) \abs{\spaceDeriv \phi(x,t)} \d t \d x \\
			&\hspace*{2cm} = \int\limits_{D^{\pm} \times \bbT} L_I \Big( \sign(\nu(\spaceVar, \timeVar) - \stochSteadyStateSolutions)(\nu(\spaceVar, \timeVar) - \stochSteadyStateSolutions) \\
			&\hspace*{4cm}  - \sign(\nu(\spaceVar, \timeVar) - \stochSteadyStateSolutions)(\tilde{\nu}(\spaceVar, \timeVar) - \stochSteadyStateSolutions) \Big) \abs{\spaceDeriv \phi(x,t)} \d t \d x \\
			&\hspace*{2cm} = L_I \int\limits_{D^{\pm} \times \bbT} \sign(\nu(\spaceVar, \timeVar) - \stochSteadyStateSolutions) \\
			&\hspace*{4cm} \cdot \Big( (\nu(\spaceVar, \timeVar) - \stochSteadyStateSolutions) - (\tilde{\nu}(\spaceVar, \timeVar) - \stochSteadyStateSolutions) \Big) \abs{\spaceDeriv \phi(x,t)} \d t \d x \\
			&\hspace*{2cm} = L_I \int\limits_{D^{\pm} \times \bbT} \sign(\nu(\spaceVar, \timeVar) - \stochSteadyStateSolutions) \Big( \nu(\spaceVar, \timeVar) - \tilde{\nu}(\spaceVar, \timeVar) \Big) \abs{\spaceDeriv \phi(x,t)} \d t \d x \\
			&\hspace*{2cm} \leq L_I \int\limits_{D^{\pm} \times \bbT} \Big| \nu(\spaceVar, \timeVar) - \tilde{\nu}(\spaceVar, \timeVar) \Big| \abs{\spaceDeriv \phi(x,t)} \d t \d x \\
			&\hspace*{2cm} \leq C_{\phi} L_I \norm{\nu - \tilde{\nu}}_{L^{\infty}(D^{\pm} \times \bbT)}.
		\end{aligned}
	\end{equation}
	
	Consequently, the continuity estimation \eqref{eq:continuityEstimate} reduces to
	\begin{equation}
		\begin{aligned}
			\abs{G^{\alpha}(\omega, \nu) - G^{\alpha}(\omega, \tilde{\nu})} &\leq C_{\phi} L_I \norm{\nu - \tilde{\nu}}_{L^{\infty}(D^{=} \times \bbT)} + 2 C_{\phi} L_I \norm{\nu - \tilde{\nu}}_{L^{\infty}(D^{\entropyConstant} \times \bbT)} + C_{\phi} L_I \norm{\nu - \tilde{\nu}}_{L^{\infty}(D^{\pm} \times \bbT)} \\
			&\leq 4 C_{\phi} L_I \norm{\nu - \tilde{\nu}}_{L^{\infty}(\bbR \times \bbT)} ,
		\end{aligned}
	\end{equation}
	and thus, 
	\begin{equation}
		G^{\alpha}(\omega, \nu) \coloneqq \int \sign (\nu(\spaceVar, \timeVar) - \stochSteadyStateSolutions) (\fluxFunction(\stochVar, \spaceVar, \nu(\spaceVar, \timeVar)) - \steadyStatePar) \spaceDeriv \phi \d t \d x
	\end{equation}
	is continuous in $\nu \in \cX$. 
	
	\textit{Step 2:}
	It remains to show the measurability of $J^{\alpha}$ with resprect to $\omega \in \Omega$.
	Therefore, we first establish that all integrands in Equation \eqref{eq:entropyFunctional} are measurable: \\
	Let $\alpha \in [\imageBound, \infty)$ (or $(-\infty, \imageBound]$) be arbitrary but fixed.
	Note that by Corollary \ref{cor:measurability_steadyState}, the steady state solutions $\stochSteadyStateSolutions$ are measurable as a mapping $\stochVar \mapsto \stochSteadyStateSolutions$.
	Hence, the integrand of \eqref{eq:entropyIntegral-1}, i.e., $\abs{\nu(\spaceVar, \timeVar) - \stochSteadyStateSolutions} \timeDeriv \phi(x, t)$ is measurable.
	Since $\stochSteadyStateSolutions$ is measurable, the measurability of the integrand of \eqref{eq:entropyIntegral-2} is an immediate consequence of Assumption \ref{as:B-4}.
	Finally, the integrand of \eqref{eq:entropyIntegral-3} is measurable, since the steady state solutions $\stochSteadyStateSolutions$ are measurable and $\initialCond \in L^{p}(\Omega, L^{\infty}(\bbR))$ is measurable by assumption.
	
	Now, since $\phi \in C^{\infty}_{c}(\bbR \times \bbT)$ has compact support we can write $J^{\alpha}$ as
	\begin{subequations}
		\label{eq:functional_compactSet}
		\begin{align}
			(\stochVar, \nu) &\mapsto \int\limits_{\supp{\phi}} \abs{\nu(\spaceVar, \timeVar) - \stochSteadyStateSolutions} \timeDeriv \phi(x, t) \d t \d x \\
			&\qquad+ \int\limits_{\supp{\phi}} \sign (\nu(\spaceVar, \timeVar) - \stochSteadyStateSolutions) (\fluxFunction(\stochVar, \spaceVar, \nu(\spaceVar, \timeVar)) - \steadyStatePar) \spaceDeriv \phi(x, t) \d t \d x \\
			&\qquad\qquad + \int\limits_{\supp{\phi}} \abs{\initialCond(\stochVar, \spaceVar) - \stochSteadyStateSolutions} \phi (\spaceVar, 0) \d x .
		\end{align}
	\end{subequations}
	Taking the integral over $\supp{\phi}$ is a bounded linear operator, because $\supp{\phi}$ is compact. 
	Since linear operators are bounded if and only if they are continuous, we know that taking integrals over a compact domain is a continuous operation.
	Now, since the composition of a continuous function with a measurable function is measurable by {\cite[Lemma 4.22]{Aliprantis2006Infinite}}, we obtain the measurability of all three integrals in $\stochVar \in \stochDomain$.
	Therefore, the measurability of $J^{\alpha}$ in $\stochVar \in \stochDomain$ follows and the proof is complete.
\end{proof}

We are now able to state the main result about the measurability of the solution map $\solution: \stochDomain \rightarrow \solutionSpace$.
\begin{theorem}[Measurability of the stochastic entropy solution]
	\label{thm:measurability}
	Let the flux function $\fluxFunction$ satisfy Assumptions \ref{as:stochFlux} and \ref{as:stochFlux_measurability} and let the initial condition satisfy $\initialCond \in L^p (\stochDomain, L^1(\bbR))$.
	Then, the mapping $\solution: \stochDomain \rightarrow \solutionSpace, \stochVar \mapsto \solution(\stochVar, \cdot, \cdot)$ is strongly measurable, which means that it is measurable and $\solution (\stochDomain) \subset \solutionSpace$ is separable.
\end{theorem}

\begin{proof}
	For $N \in \bbN$, define the space
	\begin{equation}
		S_N \coloneqq \set{\phi \in C^{\infty}(\bbR \times \bbT; \bbR_{\geq0}): \supp(\phi) \subseteq \set{(x, t) \in \bbR \times \bbT: \abs{x}, t \leq N}} .
	\end{equation}
	Each space $S_N$ is a subspace of $C^{\infty}(\bbR \times \bbT; \bbR_{\geq0})$ and thus has a countable basis, since $C^{\infty}$ has a countable basis.
	Note that for every $\phi \in C^{\infty}_{c}(\bbR \times \bbT; \bbR_{\geq0})$, there exists $S_K$, $K \in \bbN$, such that $\phi \in S_K$.
	
	Now, for fixed $N \in \bbN$, let $(\varphi^{i}_{N}, i \in \bbN) \subset S_N$  be a basis of $S_N$.
	For fixed $i \in \bbN$ and $\steadyStatePar \in [\imageBound, \infty)$ (or $\steadyStatePar \in (-\infty, \imageBound]$), define the functional
	\begin{align*}
		J^{\alpha}_{i, N} &: \stochDomain \times \cX \rightarrow \bbR , \\
		(\stochVar, \nu) &\mapsto \int \abs{\nu(\spaceVar, \timeVar) - \stochSteadyStateSolutions} \timeDeriv \varphi^{i}_{N} (x,t) \d t \d x \\
		&\hspace*{2cm} + \int \sign (\nu(\spaceVar, \timeVar) - \stochSteadyStateSolutions) (\fluxFunction(\stochVar, \spaceVar, \nu(\spaceVar, \timeVar)) - \steadyStatePar) \spaceDeriv \varphi^{i}_{N} (x,t) \d t \d x \\
		&\hspace*{3cm} + \int \abs{\initialCond(\stochVar, \spaceVar) - \stochSteadyStateSolutions} \varphi^{i}_{N} (\spaceVar, 0) \d x .
	\end{align*}
	
	Due to Proposition \ref{prop:EntropyFunctional_Caratheodory}, the functional $J^{\alpha}_{i, N}$ is a Carathéodory map, i.e., it is measurable w.r.t. $\omega \in \stochDomain$ and continuous w.r.t. $w \in \cX$.
	Note, this implies $-J^{\alpha}_{i, N}$ being a Carathéodory map.
	To show the measurability of $\solution$, we define the set-valued map $\Xi^{\alpha}_{i, N}: \stochDomain \rightrightarrows \cX$
	\begin{align*}
		\Xi^{\alpha}_{i, N}(\stochVar) = \{ \nu \in \cX | J^{\alpha}_{i, N}(\omega, \nu) \geq 0 \} , \text{ for } \stochVar \in \stochDomain .
	\end{align*}
	By {\cite[Prop. 6.3.4]{Geletu2006Introduction}}, this set-valued map $\Xi^{\alpha}_{i, N}(\stochVar)$ is measurable, since $\cX = \solutionSpace$ is separable.
	
	Define now the set-valued map $\Xi^{\alpha}: \stochDomain \rightrightarrows \cX$ via
	\begin{equation}
		\Xi^{\alpha}(\stochVar) = \bigcap_{i \in \bbN} \bigcap_{N \in \bbN} \Xi^{\alpha}_{i, N}(\stochVar) ,
	\end{equation}
	which is measurable as a countable intersection of measurable maps.
	
	We are now able to define the set-valued map $\Xi: \stochDomain \rightrightarrows \cX$ as
	\begin{equation}
		\Xi(\stochVar) = \bigcap_{\alpha \in [\imageBound, \infty) \cap \bbQ} \Xi^{\alpha}(\stochVar) ,
	\end{equation}
	which is also measurable as a countable intersection of measurable maps.
	Note, if $\steadyStatePar \in (-\infty, \imageBound]$, we need to intersect over $\steadyStatePar \in (-\infty, \imageBound] \cap \bbQ$ instead of $\alpha \in [\imageBound, \infty) \cap \bbQ$.
	
	Note, since the adapted entropy functionals $J^{\alpha}_{i, N}$ depend continuously on $\alpha \in [\imageBound, \infty)$ it is sufficient to only consider $\alpha \in [\imageBound, \infty) \cap \bbQ$ here, since the rational numbers are dense in $\bbR$.
	To see this, let us first note that the intersection over $\alpha$ is necessary to select the adapted entropy solution, which satisfies the adapted entropy condition for all $\alpha \in [\imageBound, \infty)$. 
	Let us now assume, that for $\alpha \in [\imageBound, \infty) \cap \bbQ$ the set $\Xi^{\alpha}(\stochVar)$ does contain the adapted entropy solution $u$ and an additional function $\tilde{u}$, satisfying the entropy condition for these $\alpha$.
	This implies, as a direct consequence of the existence and uniqueness theorem \ref{thm:stochExistenceUniqueness}, that for $\gamma \in [\imageBound, \infty) \setminus \bbQ$, the image of $\Xi^{\gamma}(\stochVar)$ does only contain the adapted entropy solution. 
	By definition of the adapted entropy functional, this means that there exist $i_{-}, N_{-} \in \bbN$ such that
	\begin{align*}
		J^{\alpha}_{i, N}(\stochVar, \tilde{u}) &\geq 0 \qquad \forall i, N \in \bbN , \\
		J^{\gamma}_{i_{-}, N_{-}}(\stochVar, \tilde{u}) &< 0 \qquad \text{for } i_{-}, N_{-} \in \bbN .
	\end{align*}
	Since $J^{\gamma}_{i_{-}, N_{-}}$ is continuous in $\gamma$, we know that there exist an $\varepsilon$-neighbourhood $\cU_{\varepsilon}$ of $J^{\gamma}_{i_{-}, N_{-}}(\stochVar, \tilde{u})$ such that for every $\xi \in \cU_{\varepsilon}$, we have $\xi < 0$.
	From $\bbQ$ being dense in $\bbR$ and the continuity of $J^{\gamma}_{i_{-}, N_{-}}$ in $\gamma$, we know that there exist $\delta > 0$ and $\tilde{\gamma} \in [\imageBound, \infty) \cap \bbQ$, such that $\abs{\gamma - \tilde{\gamma}} < \delta$ and $J^{\tilde{\gamma}}_{i_{-}, N_{-}}(\stochVar, \tilde{u}) \in \cU_{\varepsilon}$. 
	However, $J^{\tilde{\gamma}}_{i_{-}, N_{-}}(\stochVar, \tilde{u}) \in \cU_{\varepsilon}$ implies that $J^{\tilde{\gamma}}_{i_{-}, N_{-}}(\stochVar, \tilde{u}) < 0$. 
	But this is a contradiction to $J^{\alpha}_{i, N}(\stochVar, \tilde{u}) \geq 0$ for all $i, N \in \bbN$ and all $\alpha \in [\imageBound, \infty) \cap \bbQ$.
	Consequently, it is sufficient to intersect over $\alpha \in [\imageBound, \infty) \cap \bbQ$ to select the adapted entropy condition.
	
	Due to the pathwise existence and uniqueness theorem \ref{thm:stochExistenceUniqueness}, we know that for fixed $\stochVar \in \stochDomain$, $\Xi(\stochVar)$ only contains the pathwise unique adapted entropy solution of \eqref{eq:stochasticConservationLaw}.
	Thus, the solution map $\solution: \stochDomain \rightarrow \cX$ is measurable.
\end{proof}

\section{Stochastic jump coefficient}
\label{sec:stochJumpCoeff}

In this section, we focus on the specific choice of a multiplicative flux function (see also Example \ref{ex:multiplicativeFlux} and \ref{ex:multiplicativeFlux_measurability}), where the stochastic coefficient models heterogeneities in a medium.
For this jump coefficient, we use the random coefficient introduced in \cite{Barth2018study, Barth2020Numerical} for an elliptic diffusion problem.
It consists of a (spatial) Gaussian random field with additive discontinuities on random submanifolds.
Since the coefficient is intended to model fractured or heterogeneous media, it is assumed to be time-independent. \\
In general, it will not be possible to sample the coefficient exactly.
Therefore, we also introduce suitable approximations of the coefficient.

\begin{definition}[Stochastic jump coefficient]
	\label{def:stochasticJumpCoefficient}
	
	Let $\domainOfInterest \subset \bbR$ be a compact subset. 
	We define the stochastic jump coefficient as a function
	\begin{equation}
		\label{eq:stochasticJumpCoefficient}
		\begin{aligned}
			\jumpCoeff: \stochDomain \times \bbR &\rightarrow \bbR_{> 0} \ , \\
			(\stochVar, \spaceVar) &\mapsto \meanField(\spaceVar) + \gaussFunctional (\gaussField(\stochVar, \spaceVar)) + \jumpField (\stochVar, \spaceVar) \ ,
		\end{aligned}
	\end{equation}
	where
	\begin{itemize}
		\item $\meanField \in C(\bbR; \bbR_{\geq 0})$ is a deterministic, uniformly bounded mean function.
		\item $\phi \in C^{1}(\bbR; \bbR_{>0})$ is a continuously differentiable, positive mapping.
		\item For a (zero-mean) Gaussian random field $\gaussField_{\bbR} \in L^2(\stochDomain; L^2(\bbR))$ associated to a non-negative, symmetric trace class (covariance) operator $\covarianceOperator: L^2(\bbR) \rightarrow L^2(\bbR)$, the truncated Gaussian random field $\gaussField \in L^{2}(\stochDomain; L^2(\bbR))$ is defined as
		\begin{equation}
			\label{eq:truncatedGaussField}
			\begin{aligned}
				\gaussField (\stochVar, \spaceVar) = 
				\begin{cases}
					\gaussField_{\bbR} (\stochVar, \spaceVar) , & \spaceVar \in \domainOfInterest \\
					\min \left( \gaussField_{\bbR} (\stochVar, \spaceVar), \sup_{\spaceVar \in \domainOfInterest} \gaussField_{\bbR} (\stochVar, \spaceVar) \right) , & \spaceVar \in \bbR \setminus \domainOfInterest .
				\end{cases}
			\end{aligned}
		\end{equation}
		\item $\partition: \stochDomain \rightarrow \borelSigmaAlgebra(\domainOfInterest), \ \stochVar \mapsto \{ \partition_1, \ldots, \partition_{\numberOfJumps} \}$ is a random partition of $\domainOfInterest$, i.e., the $\partition_i$ are disjoint open subsets of $\domainOfInterest$ with $\overline{\domainOfInterest} = \bigcup_{i=1}^{\numberOfJumps} \overline{\partition_i}$. The number of elements in $\partition$ is a random variable $\numberOfJumps: \stochDomain \rightarrow \bbN$ on $\probSpace$.\\
		For $\leftDomainBound$ and $\rightDomainBound$ being the left and right boundary of $\domainOfInterest$, respectively, we define $\partition_0 := (-\infty, \leftDomainBound)$ and $\partition_{\numberOfJumps+1} := (\rightDomainBound, +\infty)$.
		\item A measure $\jumpMeasure$ on $\domainOfInterest, \borelSigmaAlgebra(\domainOfInterest)$ is associated to $\partition$ and controls the position of the random elements $\partition_i$.
		\item $(\jumpHeights{i}, i \in \bbN_{0})$ is a sequence of random variables on $\probSpace$ with arbitrary positive distribution(s) satisfying $\jumpHeights{i} < \infty$ for every $\stochVar \in \stochDomain$.
		Further, the sequence $(\jumpHeights{i}, i \in \bbN_{0})$ is independent of $\numberOfJumps$ (but not necessarily i.i.d.) and we have
		\begin{equation}
			\label{eq:defJumpField}
			\jumpField: \stochDomain \times \domainOfInterest \rightarrow \bbR_{>0}, \quad (\stochVar, \spaceVar) \mapsto \sum_{i = 0}^{\numberOfJumps+1} \boldsymbol{1}_{\partition_i}(x) \jumpHeights{i}(\stochVar) \ .
		\end{equation}
	\end{itemize}
\end{definition}
\begin{remark}
	We note that we do not require the truncated Gaussian random field $W$ and the jump field $P$ to be stochastically independent.
	
	Further, the measure $\jumpMeasure$ associated to $\partition$ does not only affect the average number of partition elements $\bbE (\numberOfJumps)$, but also the size of the partition elements $\partition_i$.
\end{remark}

\begin{corollary}[Assumptions on stochastic jump coefficient are satisfied]
	The stochastic jump coefficient defined in Definition \ref{def:stochasticJumpCoefficient} satisfies the assumptions for the stochastic flux function, see Example \ref{ex:multiplicativeFlux}.
\end{corollary}

\begin{proof}
	To satisfy the assumptions for a multiplicative flux function, the stochastic coefficient $\jumpCoeff (\stochVar, \spaceVar)$ has to be continuous in $\spaceVar \in \bbR$ except on a closed set $\nullSet$ of measure zero (which might depend on $\stochVar \in \stochDomain$).
	Further, the coefficient needs to have positive spatially bounded paths, where the bounds might be stochastic. 
	
	Since the number of elements $\numberOfJumps$ in the partition $\partition$ is finite for every $\stochVar \in \stochDomain$, the jump field $\jumpField (\stochVar, \cdot)$ is a piecewise constant function with finitely many disconitnuities.
	Further, since $\domainOfInterest$ is compact, the set $\mathfrak{D}(\stochVar)$ of discontinuities is closed and has measure zero.
	Thus, since the deterministic mean function $\meanField$ is continuous by assumption and we can always consider a continuous modification of the truncated Gaussian random field $\gaussField$, it follows that $\jumpCoeff (\stochVar, \spaceVar)$ is continuous in $\spaceVar \in \bbR$, except on a closed set $\mathfrak{D}(\stochVar, \cdot)$ of measure zero.
	
	It remains to show that the coefficient $\jumpCoeff$ has positive spatially bounded paths:
	We have that $\jumpHeights{i} > 0$ for all $i \in \bbN_0$ by construction. 
	Thus, there exists a $\coeffLowerBound (\stochVar) > 0$, such that $\jumpField(\stochVar, \spaceVar) \geq \coeffLowerBound (\stochVar) > 0$ for every $\spaceVar \in \bbR$.
	Since $\meanField$ and $\phi$ are nonnegative, we have that $\jumpCoeff(\stochVar, \spaceVar) \geq \coeffLowerBound (\stochVar) > 0$
	
	Further, since $\numberOfJumps (\stochVar) < \infty$, for each $\stochVar \in \stochDomain$ there exists $P_{+} (\stochVar) < \infty$ with $\jumpField(\stochVar, \spaceVar) \leq \max_{0 \leq i \leq \numberOfJumps (\stochVar)+1} \jumpHeights{i}(\stochVar) \eqqcolon P_{+} (\stochVar)$.
	Since the Gaussian random field $\gaussField_{\bbR}(\stochVar, \spaceVar)$ is continuous, it is bounded on $\domainOfInterest$.
	Then, by construction, it follows that there exists $W_{+} < \infty$ such that $\gaussField (\stochVar, \spaceVar) \leq W_{+}$.
	Since the mean function $\meanField$ is bounded by assumption, we have shown that the stochastic jump coefficient from Definition \ref{def:stochasticJumpCoefficient} has positive spatially bounded paths, which proves the assertion.
\end{proof}
\begin{remark}[Alternative construction of the Gaussian random field]
	\label{rem:alternativeGaussFieldConstruction}
	The rather complicated construction of the random field $\gaussField$ via the domain of interest $\domainOfInterest$ is necessary to ensure the boundedness of the spatial paths of the coefficient.
	Alternatively, it is possible to replace $\phi(\gaussField(\stochVar, \spaceVar))$ by $\phi(\gaussField_{\bbR}(\stochVar, \spaceVar))$ in Equation \eqref{eq:stochasticJumpCoefficient}, if one additionally assumes $\phi \in C^{1}(\bbR; \bbR_{>0})$ to be bounded.
\end{remark}

\subsection{Numerical approximation of the stochastic jump coefficient}

In general, the structure of the stochastic jump coefficient does not allow us to sample from its exact distribution. 
To approximate the Gaussian field $\gaussField_{\bbR}$, one can use the Karhunen-Lo\`{e}ve expansion:\\
Let $((\eigenValue_{i}, \eigenFunction_{i}), i \in \bbN)$ denote a sequence of eigenpairs of the covariance operator $\covarianceOperator$. 
Here, we assume the eigenvalues to be given in a decaying order, i.e., $\eigenValue_{1} \geq \eigenValue_{2} \geq \cdots \geq 0$.
As the covariance operator is trace class, the Gaussian random field $\gaussField_{\bbR}$ admits the representation
\begin{equation}
	\label{eq:KarhunenLoeve}
	\gaussField_{\bbR} = \sum_{i \in \bbN} \sqrt{\eigenValue_{i}} \eigenFunction_{i} \KLRV_{i} \ ,
\end{equation}
where $(\KLRV, i \in \bbN)$ is a sequence of independent standard normally distributed random variables.
An obvious approximation of the random field $\gaussField_{\bbR}$ is then given by the truncated Karhunen-Lo\`{e}ve expansion $\approxGaussField$ of $\gaussField_{\bbR}$, which is defined as
\begin{equation}
	\label{eq:truncatedKarhunenLoeve}
	\approxGaussField := \sum_{i = 1}^{\KLindex} \sqrt{\eigenValue_{i}} \eigenFunction_{i} \KLRV_{i} \ .
\end{equation}
Here, $\KLindex \in \bbN$ is called the \textit{cut-off index} of $\approxGaussField$.
Now, with this approximation of $\gaussField_{\bbR}$, we can apply the truncation \eqref{eq:truncatedGaussField} to obtain an approximation of the truncated Gaussian random field $\gaussField$.
Note, if we use the construction proposed in Remark \ref{rem:alternativeGaussFieldConstruction} for $\gaussField$, instead of applying the truncation to $\gaussField_{\bbR}$, we can directly use the truncated Karhunen-Lo\'{e}ve expansion $\approxGaussField$ as an approximation.

\section{Numerical Experiments}

In this section, we demonstrate, how the unknown $\solution$ may be approximated. 
Therefore, in Section \ref{ssec:adaptiveDiscretization}, we introduce a sample-adapted Finite Volume scheme. 
Here, sample-adaptivity means that the Finite Volume mesh is aligned \textit{a-priori} with the discontinuities of the stochastic jump coefficient $\jumpCoeff$. 
Consequently, the resulting discretization is stochastic, i.e., the mesh changes for each $\stochVar \in \stochDomain$.
This is in contrast to classical \textit{adaptive} Finite Volume methods, which are based on a-posteriori error estimates and remeshing the grid various times for each sample.

Afterwards, we investigate the approximations of the solution of the stochastic conservation law \eqref{eq:stochasticConservationLaw} with the specific flux function proposed in Example \ref{ex:multiplicativeFlux} combined with the stochastic jump coefficient introduced in Section \ref{sec:stochJumpCoeff}.
Therefore, we conduct experiments showing the performance of different numerical approximations and show the influence of different types of random fields on the pathwise solution of the stochastic conservation law. 
We conclude this section by investigating the strong error of the solution and how it is affected by the various parameters of the stochastic jump coefficient.

To the best of the author's knowledge, writing a general numerical scheme to compute the adapted entropy solution in the sense of Audusse and Perthame \cite{Audusse2005Uniqueness} is still an open problem. 
However, Godunov-type methods have been successfully applied to hyperbolic conservation laws with discontinuous flux functions for many other notions of solutions, see \cite{Adimurthi2005Godunov} for so-called solutions of type $(A, B)$ or \cite{Towers2000Convergence, Towers2018Convergence} for solutions satisfying an interface condition at the discontinuity points.
In 2019, Towers \cite{Towers2020existence} proved convergence of a finite difference scheme to the adapted entropy solution defined by Audusse \cite{Audusse2005Uniqueness} for a special class of conservation laws with flux functions in the space of bounded variation (BV).
Recently, Ghoshal et. al. \cite{Ghoshal2020Convergence} proved the convergence of a Godunov scheme to the Audusse-Perthame adapted entropy solution for the case of BV flux functions. 

While the method proposed in \cite{Ghoshal2020Convergence} might be suitable to approximate the proposed stochastic conservation law for a broad class of stochastic jump coefficients, we note that in general, we cannot ensure our flux function to have bounded variation without restricting the possible choice of the covariance operator $\covarianceOperator$.

Throughout this section, we restrict ourselves to a bounded domain $\domainOfInterest \subset \bbR$ and a finite time interval $\timeInterval = (0, T) \subset \bbR_{> 0}$.
Unless stated otherwise, we have $\domainOfInterest \times \timeInterval = (0, 1)^2$.

\subsection{Sample-adapted discretization}
\label{ssec:adaptiveDiscretization}

In this section, we first propose different meshing strategies for approximating the adapted entropy solution via a Finite Volume scheme.
The main difference of these meshing strategies is, how they account for the discontinuities in the flux function.
Afterwards, we summarize the Finite Volume discretization for the reader's convenience.

\subsubsection{Sample-adapted meshing strategies}
\label{sssec:meshing-strategies}

In this section, we propose two meshing strategies that account for the discontinuities of the flux function.
First, we describe a samplewise jump-adapted meshing. 
Such a meshing has already been used in \cite{Barth2020Numerical} for creating Finite Element meshes to approximate advection-diffusion problems.
Afterwards, we propose an improvement of this sample-adapted meshing for non-linear advection equations. 
This novel strategy accounts for the standing-wave profiles occurring due to the flux discontinuities.

\paragraph{Samplewise jump-adapted meshing}
\label{par:jump-adaptive_meshing}

Let $\stochVar \in \stochDomain$ be arbitrary but fixed and denote by $\discontinuitySet (\stochVar) \subset \spaceDomain$ the set of discontinuities of $\jumpCoeff(\stochVar, \cdot)$.
For $\numGridPoints \in \bbN$, let $\spaceGrid = \{ \gridPoint{\spaceIdx + \frac{1}{2}}\}_{\spaceIdx = 0}^{\numGridPoints}$ be a decomposition of $\spaceDomain$ with $\gridPoint{\spaceIdx - \frac{1}{2}} < \gridPoint{\spaceIdx+\frac{1}{2}}$ such that $\discontinuitySet (\stochVar) \subset \spaceGrid$.
We denote the $\spaceIdx$-th grid cell by $\gridCell{\spaceIdx} := (\gridPoint{\spaceIdx - \frac{1}{2}}, \gridPoint{\spaceIdx + \frac{1}{2}})$ and its associated spatial mesh size is defined as $\spaceStepSize{\spaceIdx} = \gridPoint{\spaceIdx+\frac{1}{2}} - \gridPoint{\spaceIdx - \frac{1}{2}}$.
Similarly, for $\numTimePoints \in \bbN$, let $\timeGrid = \{ \timePoint{\timeIdx} \}_{\timeIdx = 0}^{\numTimePoints}$ be a discretization of the time interval with time step size $\timeStepSize{\timeIdx} = \timePoint{\timeIdx + 1} - \timePoint{\timeIdx}$, for $\timeIdx = 0, \ldots, \numTimePoints - 1$.
Note, the above construction ensures that each spatial discontinuity of the flux function is aligned with a cell interface of the Finite Volume method.

\paragraph{Samplewise jump-adapted wave-cell meshing}
\label{par:wave-cell_meshing}

The samplewise jump-adapted meshing method accounts for the discontinuities in the flux function, which produce standing wave profiles in the solution.
However, on meshes which allow for large step sizes, this standing wave profile might not be approximated sufficiently, if cells next to discontinuities are large.
Let again $\stochVar \in \stochDomain$ be arbitrary but fixed and let a samplewise jump-adapted mesh be given.
We denote by $\discontinuityInterface{i} \in \spaceGrid$ the grid cell interfaces that correspond to a discontinuity of the flux function and by $\waveCell{i}^{l,r}$ its left and right grid cell.
Further, we denote by $\Delta \waveCell{i}$ the size of these discontinuity adjacent cells, which we will call \textit{wave cells}. 
If 
\begin{equation}
	\label{eq:waveCellCondition}
	\Delta \waveCell{i} \geq 2 \min_{i} \spaceStepSize{i} ,
\end{equation}
we propose to introduce another grid point, such that the new wave cell satisfies
\begin{equation}
	\label{eq:waveCellSizeCondition}
	\min_{i} \spaceStepSize{i} \leq \Delta \waveCell{i} < 2 \min_{i} \spaceStepSize{i} .
\end{equation}
As we see in the subsequent experiments, this meshing method can significantly improve the approximation of standing waves occurring in the solution.
Note, this new mesh does not change the Courant-Friedrichs-Lewy (CFL) condition of the jump-adapted mesh, as it does not affect the minimal step size of the discretization.

\subsubsection{Finite Volume method}
\label{sssec:FV-method}

In this section, we summarize the main ideas of the Finite Volume scheme as a spatial discretization.
The formulation is independent of the chosen meshing strategy of the previous section.
We employ a classical Finite Volume discretization in conservative form, which is given by
\begin{equation}
	\label{eq:FV_approximation}
	\approxSolution{\timeIdx+1}{\spaceIdx} = \approxSolution{\timeIdx}{\spaceIdx} - \frac{\timeStepSize{\timeIdx}}{\spaceStepSize{\spaceIdx}} \left( \outFlux{}{\spaceIdx} - \inFlux{}{\spaceIdx} \right) \ , \qquad 
	\approxSolution{0}{\spaceIdx} \approx \frac{1}{\spaceStepSize{\spaceIdx}} \int_{\gridCell{\spaceIdx}} \initialCond(\stochVar, \spaceVar) \d{x} \ .
\end{equation}

In the presented numerical experiments, we employ the Godunov flux for spatially dependent flux functions, which is given by
\begin{equation}
	\label{eq:GeneralizedGodunovFlux}
	\inFlux{}{\spaceIdx} = 
	\begin{cases}
		\min\limits_{\approxSolution{}{\spaceIdx-1} \leq \optVar \leq \approxSolution{}{\spaceIdx}} \fluxFunction (\stochVar, \gridPoint{\spaceIdx - \frac{1}{2}}, \optVar) \qquad \text{if } \approxSolution{}{\spaceIdx-1} \leq \approxSolution{}{\spaceIdx} \ , \\
		\max\limits_{\approxSolution{}{\spaceIdx-1} \leq \optVar \leq \approxSolution{}{\spaceIdx}} \fluxFunction (\stochVar, \gridPoint{\spaceIdx - \frac{1}{2}}, \optVar) \qquad \text{if } \approxSolution{}{\spaceIdx} \leq \approxSolution{}{\spaceIdx-1} \ .
	\end{cases}
\end{equation}

Note, in the case of a multiplicative flux function $ \fluxFunction (\stochVar, \spaceVar, \solution) = \jumpCoeff (\stochVar, \spaceVar) f(\solution)$, with $f$ being convex, we can simplify Equation \eqref{eq:GeneralizedGodunovFlux}:
\begin{equation}
	\label{eq:unimodalGodunovFlux_convex}
	\inFlux{}{\spaceIdx} = \max \left\{ \jumpCoeff(\stochVar, \gridPoint{\spaceIdx-1}) f(\max(\approxSolution{}{\spaceIdx-1}, 0)), \jumpCoeff(\stochVar, \gridPoint{\spaceIdx}) f(\min(\approxSolution{}{\spaceIdx}, 0)) \right\}.
\end{equation}
Note, when $\jumpCoeff(\stochVar, \gridPoint{\spaceIdx-1}) = \jumpCoeff (\stochVar, \gridPoint{\spaceIdx})$, Equation \eqref{eq:unimodalGodunovFlux_convex} reduces to the classical Godunov flux, which is used for conservation laws, where the flux function is not spatially dependent.

For the numerical time integration, we use the forward Euler time integration scheme, resulting the time marching equation
\begin{equation}
	\label{eq:timeMarching}
	\approxSolution{\timeIdx+1}{\spaceIdx} = \approxSolution{\timeIdx}{\spaceIdx} - \frac{\timeStepSize{\timeIdx}}{\spaceStepSize{\spaceIdx}} \left( \outFlux{\timeIdx}{\spaceIdx} - \inFlux{\timeIdx}{\spaceIdx} \right)
\end{equation}
For stability reasons, we assume that the discretizations $\spaceGrid$ and $\timeGrid$ satisfy the Courant-Friedrichs-Lewy stability condition:
\begin{equation}
	\label{eq:CFLcondition}
	\frac{\timeStepSize{\timeIdx}}{\min\limits_{\spaceIdx} \spaceStepSize{\spaceIdx}} \CFLconst < 1 , \qquad \text{for all } \timeIdx = 0, \ldots, \numTimePoints - 1
\end{equation}
where
\begin{equation}
	\label{eq:CFLconstant_general}
	\CFLconst = \max_{i} \set{\abs{\frac{\partial \fluxFunction(\stochVar, \gridPoint{\spaceIdx - \frac{1}{2}}, \approxSolution{\timeIdx}{\spaceIdx})}{\partial \approxSolution{}{}}}}
\end{equation}
in the general case, and 
\begin{equation}
	\label{eq:CFLconstant_multiplicative}
	\CFLconst = \norm{\jumpCoeff (\stochVar)}_{\infty} \cdot \max_{i} \set{\abs{\f'(\approxSolution{\timeIdx}{\spaceIdx})}}
\end{equation}
in case of a multiplicative flux function.

\subsection{Stochastic Burgers' equation}

For the numerical experiments described in this section, we consider the stochastic Burgers' equation, which is given by
\begin{equation}
	\label{eq:stochasticBurgers}
	\begin{aligned}
		\solution_{\timeVar} + \left( \jumpCoeff(\stochVar, \spaceVar) \frac{\solution^2}{2} \right)_{\spaceVar} &= 0 &&\text{in } \stochDomain \times \spaceDomain \times \timeInterval \ , \\
		\solution(\stochVar, \spaceVar, 0) &= \initialCond (\stochVar, \spaceVar) &&\text{in } \stochDomain \times \spaceDomain \times \set{0} \ ,
	\end{aligned}
\end{equation}
which is equipped with periodic boundary conditions, i.e., $\jumpCoeff (\stochVar, 0) \frac{\solution (\stochVar, 0, t)^2}{2} = \jumpCoeff (\stochVar, 1) \frac{\solution (\stochVar, 1, t)^2}{2}$, unless stated otherwise.

The deterministic mean value of the stochastic jump coefficient is given by $\meanField \equiv 0$ and we choose $\gaussFunctional (\xi) = \exp(\xi)$. 
The Gaussian field $\gaussField_{\bbR}$ is characterized by the \textit{Mat\'{e}rn covariance operator} with smoothness parameter $\smoothness > 0$, variance $\variance > 0$ and correlation length $\correlationLength > 0$, given by
\begin{equation}
	\label{eq:MaternCovariance}
	\begin{aligned}
		\covarianceOperator_{\maternAbbrev}&: L^2(\domainOfInterest) \rightarrow L^2(\domainOfInterest) \ ,\\
		[\covarianceOperator_{\maternAbbrev} \testFunction] (y) &:= \int_{\spaceDomain} \variance \frac{2^{1-\smoothness}}{\Gamma(\smoothness)} \left( \sqrt{2 \smoothness} \frac{\abs{x - y}}{\correlationLength} \right)^{\smoothness} \besselFunction \left( \sqrt{2 \smoothness} \frac{\abs{x - y}}{\correlationLength} \right) \testFunction(x) dx \quad \forall \testFunction \in L^2(\domainOfInterest) \ .
	\end{aligned}
\end{equation}
Here, $\Gamma$ denotes the Gamma function and $\besselFunction$ denotes the modified Bessel function of second kind with $\smoothness$ degress of freedom.
The spectral basis of the covariance operator $\covarianceOperator_{\maternAbbrev}$ may be approximated via Nystr\"{o}m's method \cite{Williams2006Gaussian}.

\subsection{Parameter study}

In this section, we investigate the influence of various parameters of the stochastic jump coefficient on the solution of Equation \eqref{eq:stochasticBurgers} and on the convergence rate of its approximation.
Therefore, we consider the strong $L^{1}$-error, i.e.,
\begin{equation}
	\label{eq:strongError_L1}
	\cE (\stochVar) \coloneqq \norm{\solution^{\text{ref}}_{\Delta} (\stochVar, \cdot, \stopTime) - \solution_{\Delta} (\stochVar, \cdot, \stopTime)}_{L^{1}(\bbX)} .
\end{equation}
Here, $\solution^{\text{ref}}_{\Delta}$ is a numerical reference solution computed on a finer grid and $\solution_{\Delta}$ is the considered approximation.
We first investigate the influence of the jump field. 
Afterwards, we present two experiments providing insight about the performance of the explicit and implicit time integration scheme for various simulation settings.
We conclude by showing the influence of the parameters of the Gaussian random field on the numerical approximations. \\
Note, even though we only consider the $L^{1}$-error in this parameter study, the findings qualitatively also remain true, when we consider the $L^{2}$-error instead.

\subsubsection{Jump field parameters}
\label{sssec:jumpFieldParameters}

\paragraph{Distance between jumps}
In this experiment, we investigate, how the distance between two jumps influences the performance of the finite volume approximation. 
Here, we also compare the sample-adapted discretization with the corresponding equidistant approximation. 
Therefore, we denote the jump area with size $\delta > 0$ as $J_{\delta} = (1 - \frac{\pi}{10} - \frac{\delta}{2}, 1 - \frac{\pi}{10} + \frac{\delta}{2}) \subset (0, 1)$ and define our jump coefficient as $a(x) = P(x)$, where $P(x)$ has one of the following forms:
\begin{equation}
	P^{\text{up}}(x) = \begin{cases} \frac{1}{2} & \text{for } x \in (0, 1) \setminus J_{\delta} \\ \frac{3}{2\delta} & \text{for } x \in J_{\delta} \end{cases} , \qquad P^{\text{down}}(x) = \begin{cases} \frac{3}{2} & \text{for } x \in (0, 1) \setminus J_{\delta} \\ \delta & \text{for } x \in J_{\delta} \end{cases}.
\end{equation}
In Figures \ref{fig:jumpDistance_jumpUp} and \ref{fig:jumpDistance_jumpDown}, the convergence rates of the finite volume approximation is plotted for $P^{\text{up}}$ and $P^{\text{down}}$, respectively, and for various sizes of the jump area $J_{\delta}$.
\begin{figure}[htbp!]
	\begin{subfigure}{.32\textwidth}
		\includegraphics{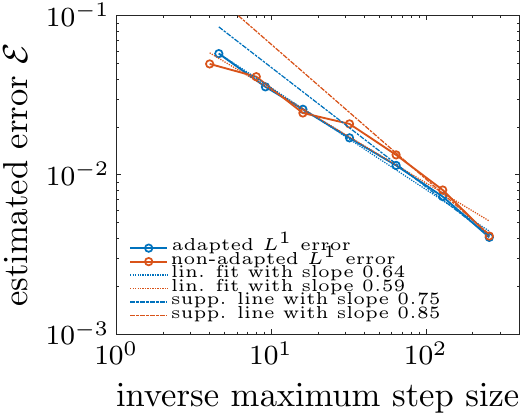}
		\caption{$\delta = 2^{-4}$.}
	\end{subfigure}
	\hfill
	\begin{subfigure}{.32\textwidth}
		\includegraphics{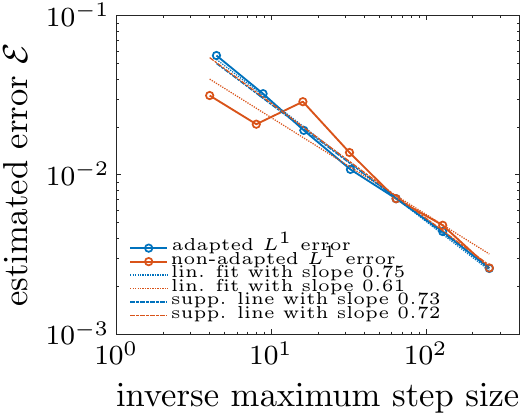}
		\caption{$\delta = 2^{-6}$.}
	\end{subfigure}
	\hfill
	\begin{subfigure}{.32\textwidth}
		\includegraphics{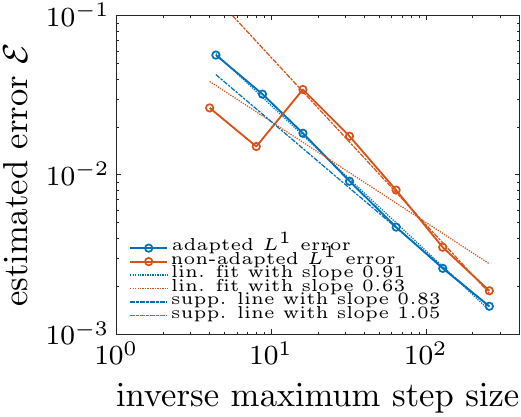}
		\caption{$\delta = 2^{-8}$.}
	\end{subfigure}
	\caption{Convergence rates for $a(x) = P^{\text{up}}(x)$ and various $\delta$.}
	\label{fig:jumpDistance_jumpUp}
\end{figure}
\begin{figure}[htbp!]
	\begin{subfigure}{.32\textwidth}
		\includegraphics{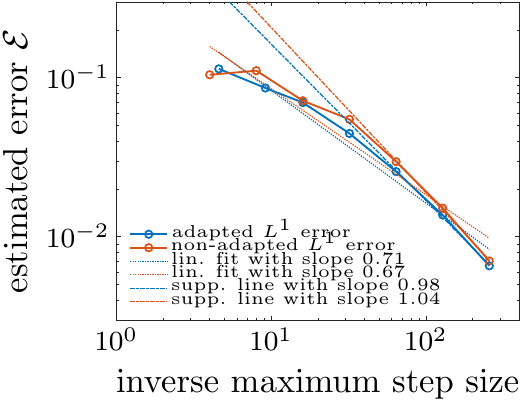}
		\caption{$\delta = 2^{-4}$.}
	\end{subfigure}
	\hfill
	\begin{subfigure}{.32\textwidth}
		\includegraphics{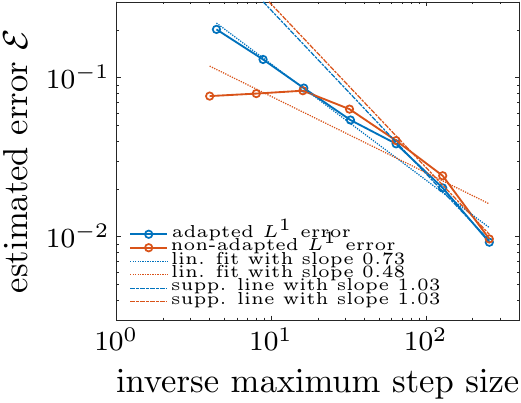}
		\caption{$\delta = 2^{-6}$.}
	\end{subfigure}
	\hfill
	\begin{subfigure}{.32\textwidth}
		\includegraphics{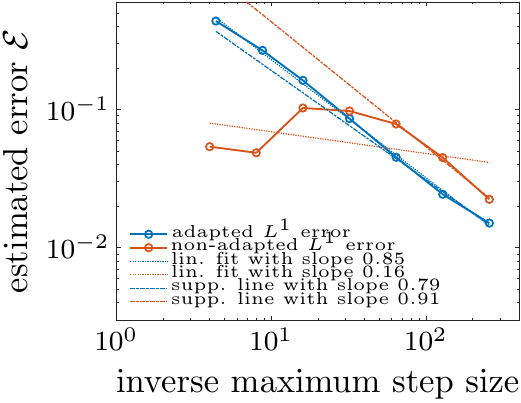}
		\caption{$\delta = 2^{-8}$.}
	\end{subfigure}
	\caption{Convergence rates for $a(x) = P^{\text{down}}(x)$ and various $\delta$.}
	\label{fig:jumpDistance_jumpDown}
\end{figure}

While the convergence rate of the jump-adapted finite volume discretization does not seem to be affected by the distance between two jumps, the equidistant finite volume method has problems with approximating small distances.

\paragraph{Number of jumps}
We continue our numerical experiments by investigating, how the number of jumps in the coefficient affects the convergent rate of the (jump-adapted) finite volume discretization. 
Therefore, we consider again the case of $a(x) = P(x)$, where $P(x)$ is defined as in Equation \eqref{eq:defJumpField}, 
\begin{equation}
	\jumpField: \domainOfInterest \rightarrow \bbR_{>0}, \quad \spaceVar \mapsto \sum_{i = 0}^{\numberOfJumps+1} \boldsymbol{1}_{\partition_i}(x) \jumpHeights{i} \ , \qquad P_i = \begin{cases} \frac{1}{2} & i \text{ odd} \\ \frac{3}{2} & i \text{ even} \end{cases} .
\end{equation}
Above, $\numberOfJumps \in \bbN$ is the number of jumps and we set the jump positions $\jumpPosition{i}$ to be uniformly distributed over the domain, i.e., $\jumpPosition{i} \sim \cU((0, 1))$.
Figure \ref{fig:numberOfJumps} shows the convergence rates for various values of $\numberOfJumps$.
\begin{figure}[htbp!]
	\begin{subfigure}{.32\textwidth}
		\includegraphics{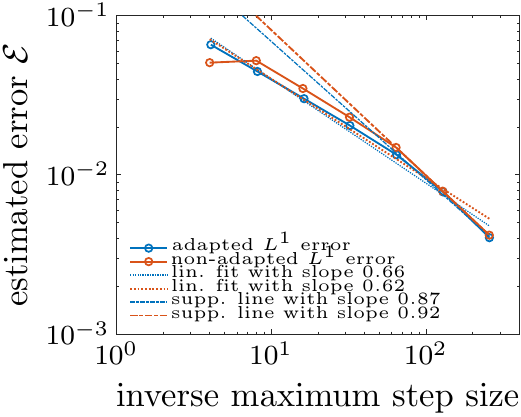}
		\caption{$\numberOfJumps = 4$ jumps.}
	\end{subfigure}
	\hfill
	\begin{subfigure}{.32\textwidth}
				\includegraphics{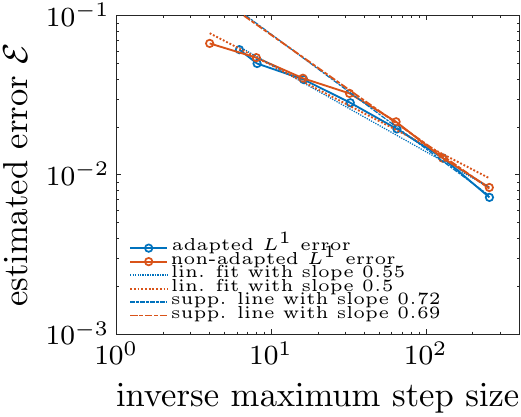}
		\caption{$\numberOfJumps = 16$ jumps.}
	\end{subfigure}
	\hfill
	\begin{subfigure}{.32\textwidth}
				\includegraphics{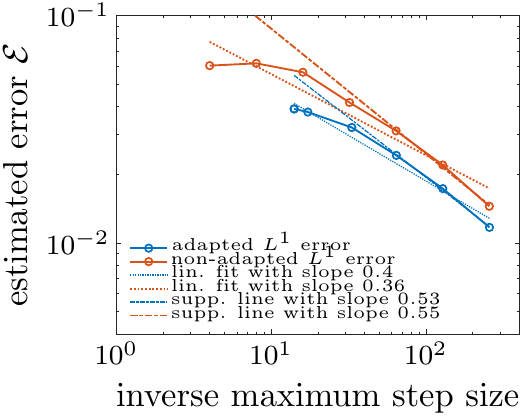}
		\caption{$\numberOfJumps = 64$ jumps.}
		\label{sfig:64jumps}
	\end{subfigure}
	\caption{Convergence rates for $a(x) = P(x)$ with varying number of jumps.}
	\label{fig:numberOfJumps}
\end{figure}

One can see that an increasing number of jumps leads to a slower convergence rate. 
Another result is that the jump-adapted finite volume approximation does not improve the convergence rate of the method. 
However, for a coefficient with many jumps, Figure \ref{sfig:64jumps} shows that the jump-adapted discretization might lead to a better constant.

\subsubsection{Explicit vs. implicit time integration}

In this section, we conduct experiments on the performance of the explicit and implicit Euler scheme to integrate the jump-adapted finite volume discretization.
Specifically, we investigate the influence of two parameters on the performance of the methods, measured in time-to-error:
(i) The influence of $\frac{\max (\Delta x)}{\min (\Delta x)}$, which influences the possible time step size due to the CFL condition in the explicit scheme and the lack thereof in the implicit scheme.
(ii) The influence of the estimated wave speed of the solution, which also influences the possible time step size.
In both experiments, we compare the behaviour for multiple discretization parameters.

\paragraph{Influence of $\frac{\max (\Delta x)}{\min (\Delta x)}$}

In this experiment, we consider the following coefficient for some parameter $\delta > 0$
\begin{equation}
	a(x) = \begin{cases}
		0.5 & \text{for } 0 < x < \frac{\pi}{5} - \frac{\delta}{2} \\
		1.5 & \text{for } \frac{\pi}{5} - \frac{\delta}{2} < x < \frac{\pi}{5} + \frac{\delta}{2} \\
		0.5 & \text{for } \frac{\pi}{5} + \frac{\delta}{2} < x < 1
	\end{cases} .
\end{equation}
Above, $\delta$ denotes the distance between the two jumps of $a$ and thus imposes a restriction on $\min (\Delta x)$. 
The corresponding time-to-error plots for various values of $\delta$ are shown in Figure \ref{fig:explciitVsImplicit_DeltaX}.
\begin{figure}[htbp!]
	\begin{subfigure}{.325\textwidth}
		\includegraphics{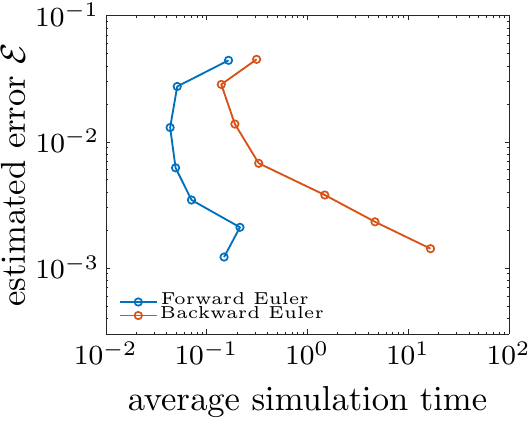}
		\caption{$\delta = 10^{-2}$.}
	\end{subfigure}
	\hfill
	\begin{subfigure}{.325\textwidth}
		\includegraphics{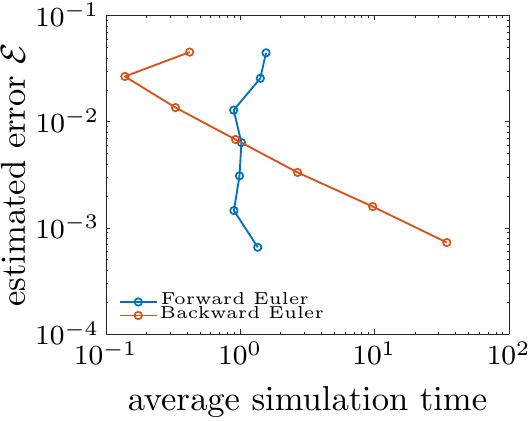}
		\caption{$\delta = 10^{-3}$.}
	\end{subfigure}
	\hfill
	\begin{subfigure}{.325\textwidth}
		\includegraphics{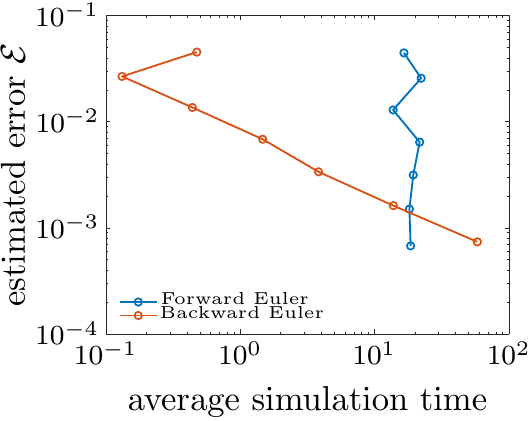}
		\caption{$\delta = 10^{-4}$.}
	\end{subfigure}
	\caption{Time to error plot of the forward Euler and backward Euler time discretization of the jump-adapted finite volume approximation for different jump distances $\delta > 0$.}
	\label{fig:explciitVsImplicit_DeltaX}
\end{figure}

It is not surprising that the backward Euler time discretization needs more time with decreasing error, since the size of the nonlinear system, which needs to be solved at each time step, increases.
Note that the computation time does not increase with finer discretizations. 
However, the amount of time necessary to evolve the solution over time with the explicit scheme depends highly on the allowed maximum time step size, which is directly influenced by $\delta$ through the CFL condition.

\paragraph{Influence of the estimated wave speed}

In this experiment, we conduct experiments on the influence of the estimated wave speed on the performance of the explicit and implicit time integrator, respectively.
Therefore, we consider an initial condition $\initialCond (x) = \kappa \sin (\pi x)$ with $\kappa \in \bbR_{>0}$ and a jump coefficient of the form
\begin{equation}
	a(x) = \begin{cases}
		P_1 & \text{for } 0 < x < \frac{\pi}{5} - \frac{1}{20} \\
		P_2 & \text{for } \frac{\pi}{5} - \frac{1}{20} < x < \frac{\pi}{5} + \frac{1}{20} \\
		P_1 & \text{for } \frac{\pi}{5} + \frac{1}{20} < x < 1
	\end{cases} .
\end{equation}
Figure \ref{fig:explicitVsImplicit_waveSpeed} visualizes the time-to-error plots for the explicit and implicit time integrator for different values of $\kappa$ and $P_1, P_2$.
\begin{figure}[h!tbp]
	\centering
	\begin{subfigure}[b]{.325\textwidth}
		\includegraphics{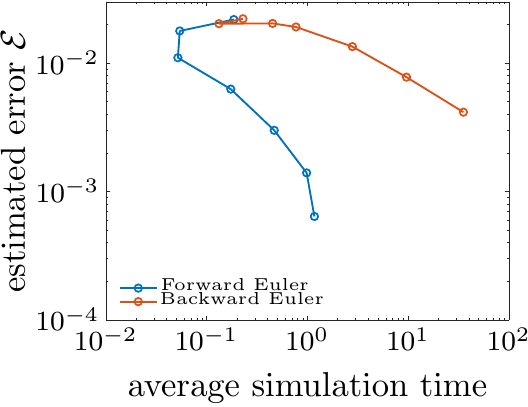}
		\caption{$\kappa = 0.3$, $P_1 = 10.5$, $P_2 = 20$.}
	\end{subfigure}
	\hfill
	\begin{subfigure}[b]{.325\textwidth}
		\includegraphics{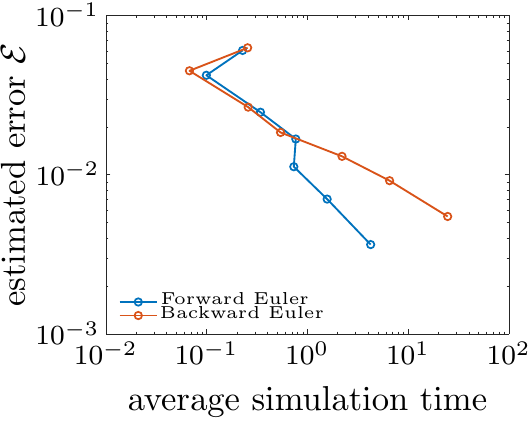}
		\caption{$\kappa = 0.3$, $P_1 = 1$, $P_2 = 50$.}
	\end{subfigure}
	\hfill
	\begin{subfigure}[b]{.325\textwidth}
		\includegraphics{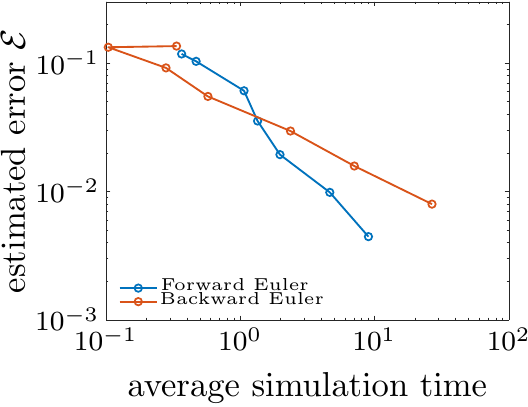}
		\caption{$\kappa = 1$, $P_1 = 1$, $P_2 = 50$.}
		\label{sfig:waveSpeed_50_sin1}
	\end{subfigure}
	\caption{Time to error plots of the forward and backward Euler scheme with a jump-adapted finite volume discretization for different values of $\kappa$ and $P_1, P_2$.}
	\label{fig:explicitVsImplicit_waveSpeed}
\end{figure}

As one can see in Figure \ref{sfig:waveSpeed_50_sin1}, the implicit time integration only outperforms the explicit time integration in the case of a coarse spatial discretization combined with a high maximum value of the jump coefficient.

\subsubsection{Parameters of Gaussian random field}

In this section, we investigate the influence of the parameters of the Gaussian random field on the convergence rates of the equidistant finite volume approximation.
Therefore, we set $a(\stochVar, \spaceVar) = \exp (W_{\bbR}(\stochVar, \spaceVar))$, i.e., we consider a log-Gaussian random field. 
The Gaussian random field $\gaussField_{\bbR}$ is characterized by the Mat{\'e}rn covariance operator (see Equation \eqref{eq:MaternCovariance}) with smoothness parameter $\smoothness > 0$, variance $\variance > 0$ and correlation length $\correlationLength > 0$.
In the subsequent experiments, we focus on the influence of the smoothness $\smoothness$ and the correlation length $\correlationLength$. \\
Since the Gaussian field is stochastic, the Figures \ref{fig:Gaussian_smoothness} and \ref{fig:Gaussian_correlationLength} show convergence rates that are averaged over $20$ samples.
For both experiments, the variance was fixed at $\variance = 0.1$. 
While the smoothness parameter $\smoothness$ was varied in the first experiment, the correlation length was fixed at $\correlationLength = 0.1$ and in the second experiment, we set $\smoothness = \infty$ and varied the correlation length.
\begin{figure}[h!tbp]
	\centering
	\begin{subfigure}[b]{.32\textwidth}
		\includegraphics{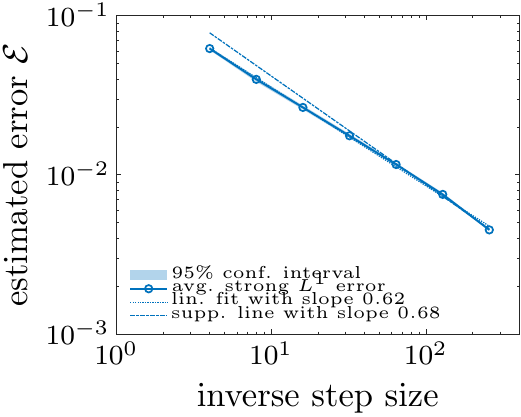}
		\caption{Mat{\'e}rn smoothness $\smoothness = 0.5$.}
	\end{subfigure}
	\hfill
	\begin{subfigure}[b]{.32\textwidth}
		\includegraphics{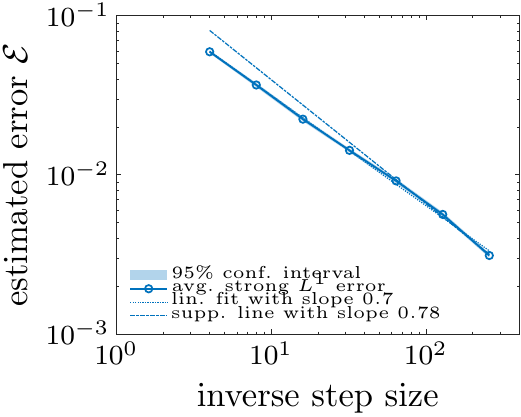}
		\caption{Mat{\'e}rn smoothness $\smoothness = 1$.}
	\end{subfigure}
	\hfill
	\begin{subfigure}[b]{.32\textwidth}
		\includegraphics{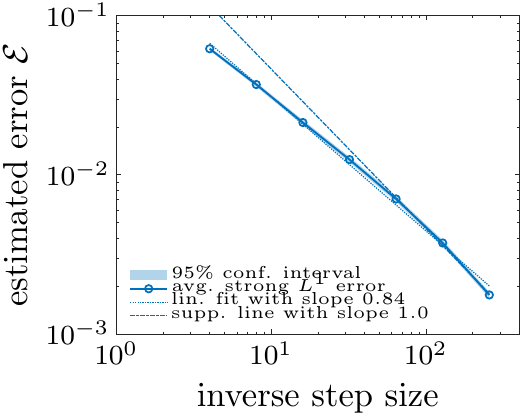}
		\caption{Mat{\'e}rn smoothness $\smoothness = \infty$.}
	\end{subfigure}
	\caption{Convergence of the finite volume approximation for different smoothness parameters $\smoothness$ of the Mat{\'e}rn covariance kernel. The error estimation is based on $20$ samples.}
	\label{fig:Gaussian_smoothness}
\end{figure}

As one can see, with decreasing smoothness parameter, the convergence rate of the finite volume approximation gets lower.
This is, what one would expect, since rough paths of the Gaussian random field are not resolved as good as smooth paths.
A similar effect can be seen by reducing the correlation length of the covariance kernel, which is shown in Figure \ref{fig:Gaussian_correlationLength}.
\begin{figure}[h!tbp]
	\centering
	\begin{subfigure}[b]{.32\textwidth}
		\includegraphics{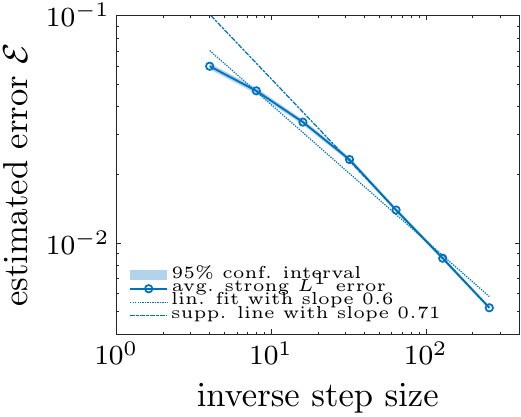}
		\caption{Correlation length $\correlationLength = 0.01$.}
	\end{subfigure}
	\hfill
	\begin{subfigure}[b]{.32\textwidth}
		\includegraphics{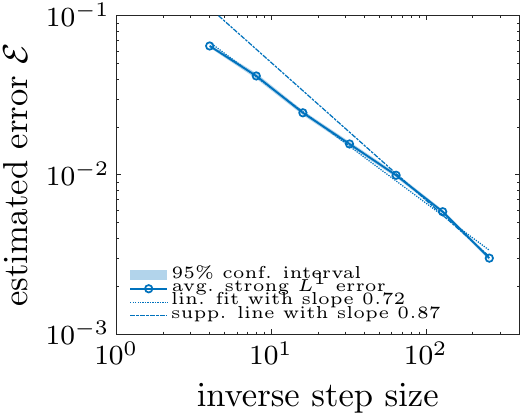}
		\caption{Correlation length $\correlationLength = 0.05$.}
	\end{subfigure}
	\hfill
	\begin{subfigure}[b]{.32\textwidth}
		\includegraphics{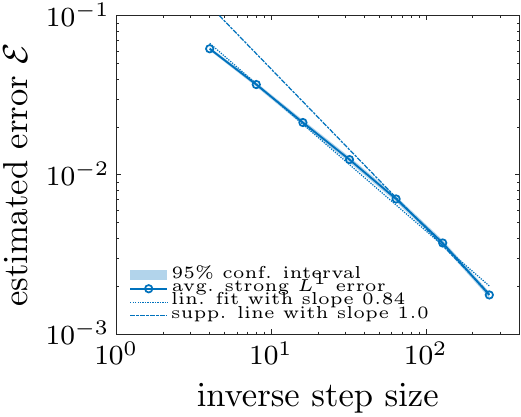}
		\caption{Correlation length $\correlationLength = 0.1$.}
	\end{subfigure}
	\caption{Convergence of the finite volume approximation for different correlation length $\correlationLength$ of the Mat{\'e}rn covariance kernel. The error estimation is based on $20$ samples.}
	\label{fig:Gaussian_correlationLength}
\end{figure}

\subsection{Pathwise convergence study}

In this section, we investigate the strong error of the jump-adapted and wave cell finite volume method when computing the stochastic Burgers equation \eqref{eq:stochasticBurgers}. 
In the experiments described in Sections \ref{sssec:alternatingJumpField} and \ref{sssec:PoissonJumpField}, we consider specific jump coefficients, which have the form described in \eqref{eq:stochasticJumpCoefficient}.
Afterwards, in Section \ref{sssec:InclusionField}, we present experiments demonstrating the ability of the wave cell finite volume method.
Therefore, we consider a stochastic jump coefficient without the Gaussian random field part, i.e., having the simpler form $\jumpCoeff (\stochVar, \spaceVar) = \meanField (\spaceVar) + \jumpField (\stochVar, \spaceVar)$.
For the convergence study, we consider both the strong $L^{1}$-error and the strong $L^{2}$-error, i.e.,
\begin{equation}
	\label{eq:strongError}
	\cE (\stochVar) \coloneqq \norm{\solution^{\text{ref}}_{\Delta} (\stochVar, \cdot, \stopTime) - \solution_{\Delta} (\stochVar, \cdot, \stopTime)}_{L^{\star}(\bbX)} .
\end{equation}
Here, $\solution^{\text{ref}}_{\Delta}$ is a numerical reference solution computed on a finer grid, $\solution_{\Delta}$ is the considered approximation and $\norm{\cdot}_{L^{\star}(\bbX)}$ denotes either the $L^{1}(\bbX)$ or the $L^{2}(\bbX)$ norm.

\subsubsection{Alternating jump field with exponential Gaussian random field}
\label{sssec:alternatingJumpField}

In this first experiment, we consider a Gaussian random field $\gaussField$ with the Matérn covariance operator \eqref{eq:MaternCovariance} with a smoothness parameter $\smoothness = \frac{1}{2}$.
The jump field $\jumpField (\stochVar, \spaceVar)$ has $\numberOfJumps \sim Poi(5)+1$ jumps, which are located at the jump locations $\jumpPosition{i} \sim \cU (\cD)$.
The jump heights $\jumpHeights{i}$ corresponding to the partition, are given by
\begin{align}
	\label{eq:alternatingJumpHeights}
	\jumpHeights{i} \sim 
	\begin{cases}
		\cU ([ \frac{1}{4}, \frac{3}{4}]) &\text{for } i \text{ odd}, \\
		\cU ([ \frac{5}{4}, \frac{7}{4}]) &\text{for } i \text{ even}.
	\end{cases}
\end{align}
To illustrate this coefficient, Figure \ref{fig:alternatingExponential_samples} shows two realizations of the coefficient.
One can see that the occurring jumps are relatively small compared to the variation of the Gaussian random field, which is influenced by the small smoothness parameter $\smoothness$, the variance $\variance$ and the correlation length $\correlationLength$.
\begin{figure}[h!tbp]
	\centering
	\begin{subfigure}[b]{0.495\textwidth}
		\centering
		\includegraphics{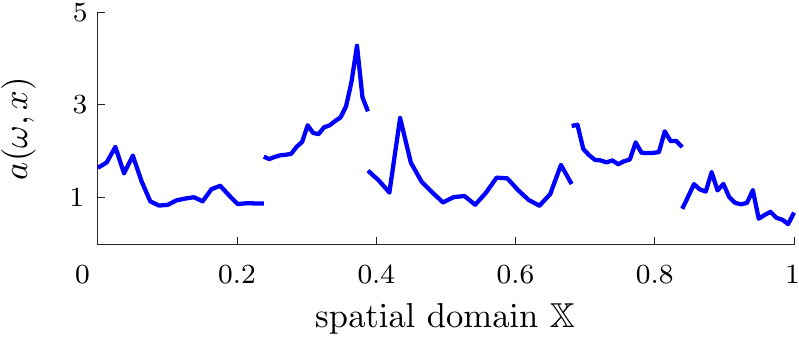}
	\end{subfigure}
	\hfill
	\begin{subfigure}[b]{0.495\textwidth}
		\centering
		\includegraphics{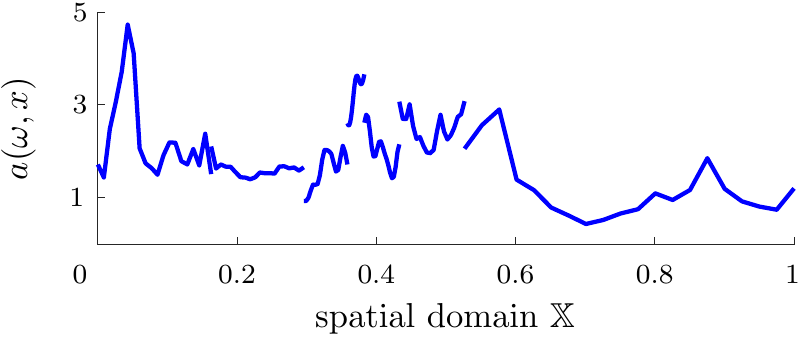}
	\end{subfigure}
	\caption{Two samples of the alternating jump field with an exponential Gaussian random field.}
	\label{fig:alternatingExponential_samples}
\end{figure}

In Figure \ref{fig:alternatingExponential_errorVsStepSize}, the pathwise $L^{1}$- and $L^{2}$-error are shown for $50$ samples. 
The corresponding reference solutions were computed with the wave cell finite volume method using a grid with approximately $4$ times as many discretization points.
One can see that the difference of the methods in not significant and that the order of convergence is similar. 
Since the distance between the jumps is quiet large on average, the discontinuities are already resolved on relatively coarse discretizations. 
Thus, the different in the discretizations has no significant effect and the presented behaviour is exactly, what we would expect.
\begin{figure}[h!tbp]
	\centering
	\begin{subfigure}[b]{.49\textwidth}
		\centering
		\includegraphics{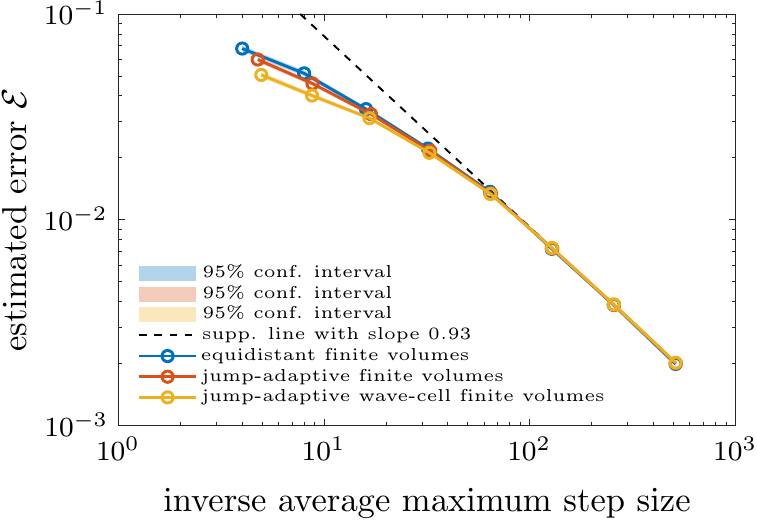}
		\caption{Pathwise $L^{1}$ error.}
	\end{subfigure}
	\hfill
	\begin{subfigure}[b]{.49\textwidth}
		\centering
		\includegraphics{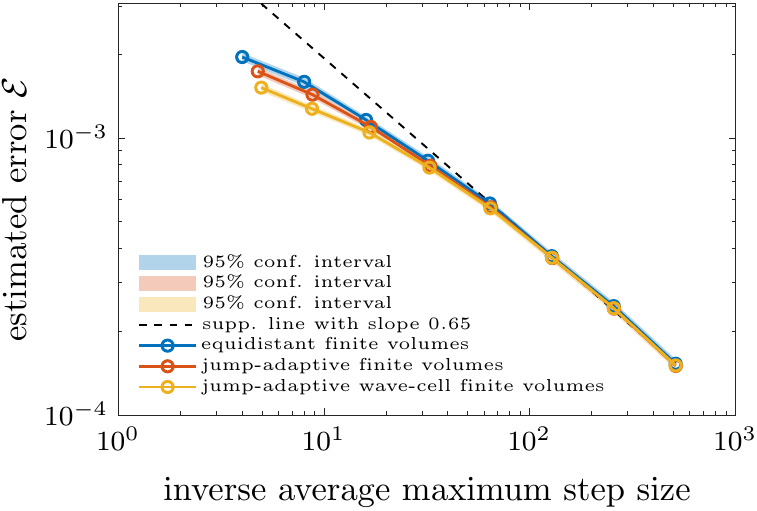}
		\caption{Pathwise $L^{2}$ error.}
	\end{subfigure}
	\caption{Strong error for the alternating jump field with exponential Gaussian random field. The error estimation is based on $50$ samples.}
	\label{fig:alternatingExponential_errorVsStepSize}
\end{figure}

\subsubsection{Poisson jump field with squared-exponential Gaussian random field}
\label{sssec:PoissonJumpField}

In this experiment, we consider a similar jump field as in Section \ref{sssec:alternatingJumpField}. 
The considered jump field $\jumpField$ has $\tau \sim Poi(5)+1$ jumps at jump positions $\jumpPosition{i} \sim \cU (\cD)$.
The corresponding jump heights are $\jumpHeights{i} \sim Poi(5) + 1$ distributed. 
This leads to a significant difference in the jump field, since the jump heights are not globally bounded. 
We additively combine this random jump field with a squared-exponential Gaussian random field, i.e., the Matérn covariance kernel \eqref{eq:MaternCovariance} having a smoothness parameter of $\smoothness = \infty$.

Two samples of this coefficient are shown in Figure \ref{fig:PoissonSquaredExponential_samples}. 
Note that, in contrast to the previous random field, the jumps are higher and have a more significant effect on the sample compared to the variation of the Gaussian part.
Also, let us mention that this coefficient is more likely to have larger values than the random field presented in the previous subsection.
\begin{figure}[h!tbp]
	\centering
	\begin{subfigure}[b]{0.495\textwidth}
		\centering
		\includegraphics{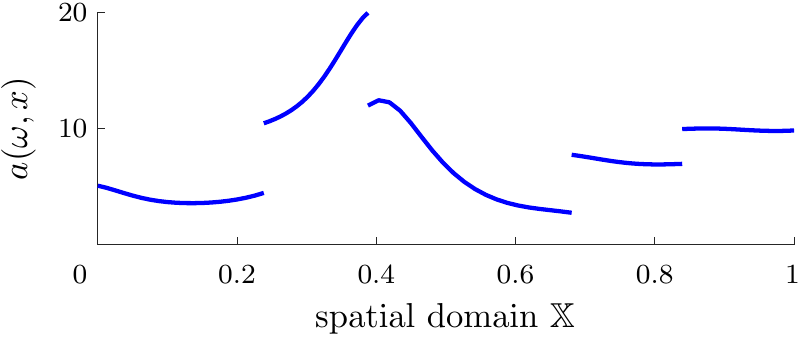}
	\end{subfigure}
	\hfill
	\begin{subfigure}[b]{0.495\textwidth}
		\centering
		\includegraphics{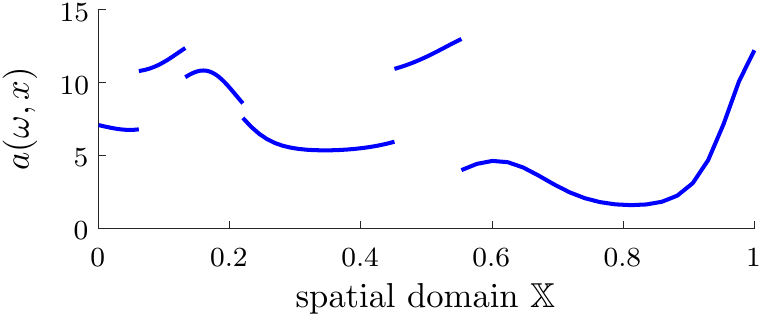}
	\end{subfigure}
	\caption{Two realizations of the Poisson jump field with squared-exponential Gaussian random field.}
	\label{fig:PoissonSquaredExponential_samples}
\end{figure}

The pathwise $L^{1}$- and $L^{2}$-error are shown in Figure \ref{fig:PoissonSquaredExponential}. 
The convergence behaviour and the absolute error are again very similar, which is as expected, since the jumps are resolved on relatively coarse grids. 
However, we note a slightly higher convergence rate for both adapted discretization methods. 
This is not surprising, since the heights of the standing shock profiles resulting from the discontinuities in the flux function are better approximated with the adapted methods.
\begin{figure}[h!tbp]
	\centering
	\begin{subfigure}[b]{.49\textwidth}
		\centering
		\includegraphics{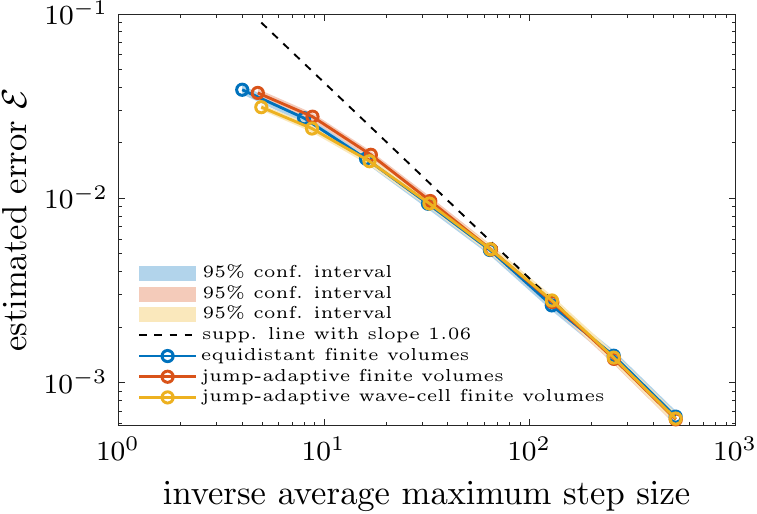}
		\caption{Pathwise $L^{1}$ error.}
	\end{subfigure}
	\hfill
	\begin{subfigure}[b]{.49\textwidth}
		\centering
		\includegraphics{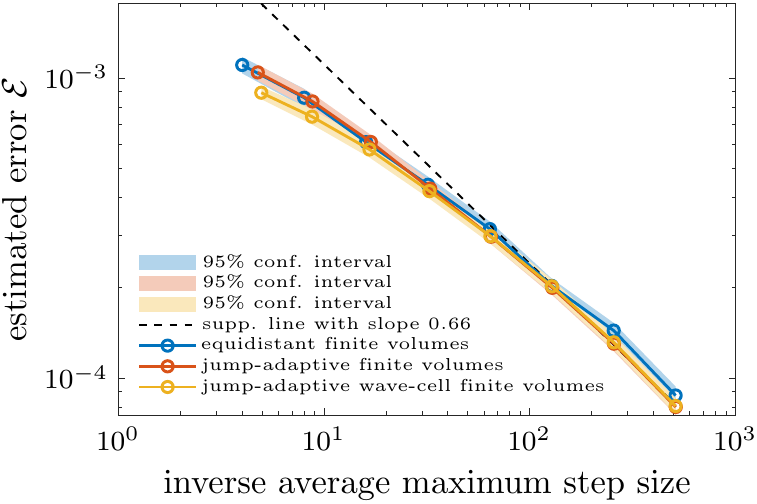}
		\caption{Pathwise $L^{2}$ error.}
	\end{subfigure}
	\caption{Strong error for the Poisson jump field with squared-exponential Gaussian random field estimated with $50$ samples.}
	\label{fig:PoissonSquaredExponential}
\end{figure}

\subsubsection{Jump field with random inclusions}
\label{sssec:InclusionField}

In this last experiment, we demonstrate the ability of the wave cell finite volume discretization. 
Therefore, we consider the following stochastic jump coefficient having random inclusions: 
Let $\tau \sim Poi(10)+1$ denote the number of inclusions with positions $\chi_i \sim \cU (\cD)$.
The corresponding inclusion sizes are given by $\delta_i \sim \cU ([10^{-5}, 10^{-3}])$ and the heights of the inclusions are given by
\begin{align}
	\label{eq:inclusionHeights}
	H_k \sim 
	\begin{cases}
		Poi(30) &\text{for } k = 0, \\
		\frac{1}{\xi}, \xi \sim Poi(30) &\text{for } k = 1.
	\end{cases}
\end{align}
Let now $\mathfrak{I} \subset \cD$ denote the union of all inclusion. 
We define the jump field with random inclusions as
\begin{align}
	\label{eq:inclusionField}
	\jumpCoeff (\stochVar, \spaceVar) = 
	\begin{cases}
		H_k &\text{for } \spaceVar \in \mathfrak{I}, \\
		1 &\text{otherwise,}
	\end{cases}
\end{align}
where $k \sim B(1, \frac{1}{2})$ is a Bernoulli-$\frac{1}{2}$-distributed random variable, defining whether the inclusion height is given by $H_1$ or by $H_2$. 

The random field is illustrated in Figure \ref{fig:inclusionField_samples}, where two samples of the coefficient are shown.
Note, the inclusions are marked with a circle to increase the visibility.
Also, we see that the character of the random field is very different to the previous ones, since no Gaussian field is involved and thus the coefficient does not vary away from the jumps.
However, due to the small size of the inclusions, this coefficient is promising to demonstrate the qualitative difference between the proposed meshing methods.
\begin{figure}[h!tbp]
	\centering
	\begin{subfigure}[b]{0.495\textwidth}
		\centering
		\includegraphics{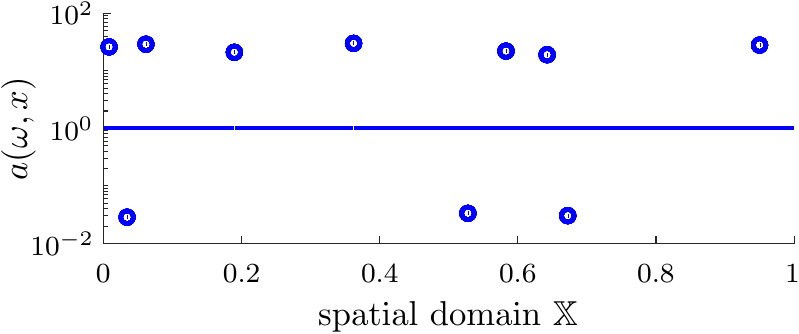}
	\end{subfigure}
	\hfill
	\begin{subfigure}[b]{0.495\textwidth}
		\centering
		\includegraphics{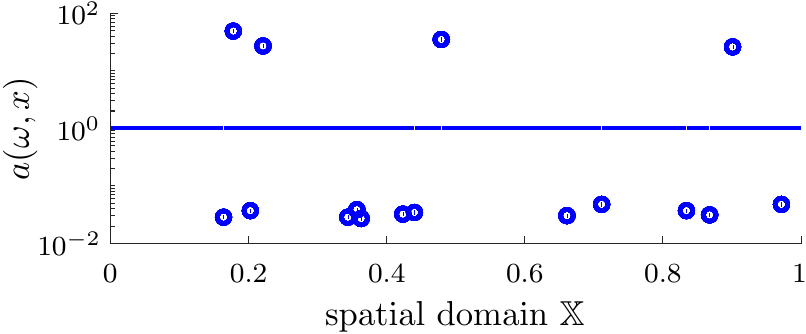}
	\end{subfigure}
	\caption{Two realizations of the jump field with random inclusions.}
	\label{fig:inclusionField_samples}
\end{figure}

In Figure \ref{fig:inclusionField}, we can see the performance of the different discretization methods measured by the strong $L^{1}$ and $L^{2}$ error for $100$ samples. 
While the adapted discretizations are converging with a similar rate, the standard finite volume method does not seem to converge at all or at a very low convergence rate. 
This behaviour is what one would expect for this jump coefficient and the standard finite volume method:
Due to the small size of the inclusions, the standard finite volume scheme does not resolve the discontinuities in the flux function. 
Thus, when refining the spatial step size, the standard finite volume discretization needs a much finer resolution of the grid to resolve the discontinuities than the adaptive discretization schemes. 
Note, asymptotically, we expect all three discretization schemes to converge with the same convergence rate.  \\
We also note the better error constant in the wave cell finite volume method compared to the jump-adapted finite volume method. 
This behaviour can be explained by the better approximation of the standing wave profiles resulting from the discontinuities of the flux function.
\begin{figure}[h!tbp]
	\centering
	\begin{subfigure}[b]{.49\textwidth}
		\includegraphics{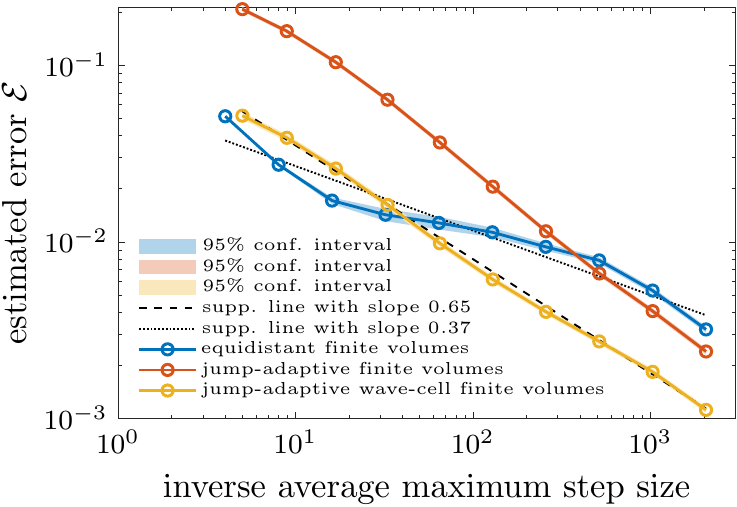}
		\caption{Pathwise $L^{1}$ error.}
	\end{subfigure}
	\hfill
	\begin{subfigure}[b]{.49\textwidth}
		\includegraphics{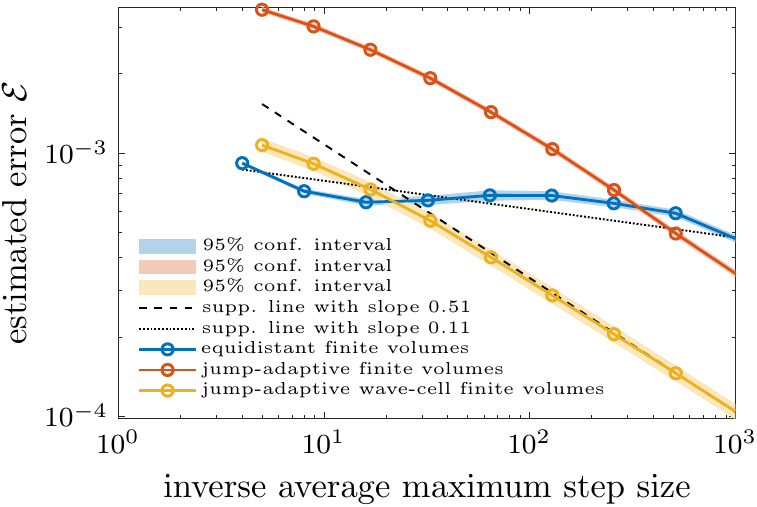}
		\caption{Pathwise $L^{2}$ error.}
	\end{subfigure}
	\caption{Strong error for the stochastic jump field with random inclusions based on $100$ samples.}
	\label{fig:inclusionField}
\end{figure}

\section*{Declaration of interest}
None.

\section*{Acknowledgement}
The research leading to these results has received funding from the German Research Foundation (Deutsche Forschungsgemeinschaft, DFG) under Germany's Excellence Strategy EXC-2075 390740016 at the University of Stuttgart and it is gratefully acknowledged.

\printbibliography
\end{document}